\documentclass[a4paper, reqno]{amsart}
\usepackage[utf8]{inputenc}
\usepackage{palatino,newpxmath} 
\usepackage[left=3cm,right=3cm,top=3cm,bottom=3cm]{geometry}
\usepackage[english]{babel}
\usepackage{natbib}
\usepackage{url}
\usepackage{amsmath}

\usepackage{amssymb}
\usepackage{amsthm}
\usepackage{mathtools}
\usepackage{amsfonts}
\usepackage{bbm}
\usepackage{caption}
\usepackage{gensymb}
\usepackage{tikz-cd}
\usepackage{mathdots}
\usepackage{cancel}
\usepackage{multicol}
\usepackage[font=small,labelfont=bf]{caption}
\usepackage[colorlinks=true, urlcolor = blue]{hyperref} 
\usepackage[shortlabels]{enumitem}
\usepackage{paracol}
\setenumerate[0]{label=(\arabic*)}

\usepackage[many]{tcolorbox}  

\newtcolorbox{boxE}{
    enhanced, 
    boxrule = 0pt, 
    borderline = {0.75pt}{0pt}{main}, 
    borderline = {0.75pt}{2pt}{sub} 
}

\newcommand{\Z}{\mathbb{Z}}

\newcommand{\R}{\mathbb{R}}
\newcommand{\C}{\mathbb{C}}
\newcommand{\D}{\mathbb{D}}

\renewcommand{\P}{\mathbb{P}}

\newcommand{\F}{\mathbb{F}}

\renewcommand{\phi}{\varphi}

\DeclareMathAlphabet{\mathpzc}{OT1}{pzc}{m}{it} 

\renewcommand{\c}{\mathfrak{c}}

\usepackage{mathrsfs}

\renewcommand{\O}{\mathscr{O}}
\newcommand{\M}{\mathcal{M}}

\newcommand{\RN}[1]{%
  \textup{\uppercase\expandafter{\romannumeral#1}}%
}

\newcommand{\Conn}{\mathfrak{Con}^V_\Theta}

\theoremstyle{definition}
\newtheorem*{defn}{Definition}
\newtheorem{es}{Example}

\newtheorem{oss}{Remark}

\newtheorem*{es*}{Example}
\newtheorem*{lem*}{Lemma}

\theoremstyle{plain}
\newtheorem{thm}{Theorem}[section]

\newtheorem{prop}[thm]{Proposition}
\newtheorem{lemma}[thm]{Lemma}

\newtheorem{theoremA}{Theorem}

\begin{document}


\title{}
\title[The Fifth Painlevé Foliation]{Compactification of the Fifth \\Painlevé Foliation}
\author{Mattia Morbello}

\begin{abstract}
   We extend the previous compactifications of the fifth Painlevé foliation studying the behaviour of the leaves for the time parameter $t$ being equal to $0$ or $\infty$. In particular we find, for each boundary component, a first integral of the Hamiltonian vector field, providing an explicit description of the foliation. 
\end{abstract}

\maketitle

\tableofcontents


\thispagestyle{empty}
\restoregeometry

\newpage

\section*{Introduction}

The fifth Painlevé equation is the complex non-linear second-order differential equation
\[\ddot q(t)=\Bigg(\frac{1}{2q(t)}+\frac{1}{q(t)-1}\Bigg)\dot q(t)^2-\frac1t\dot q(t)+\frac{(q(t)-1)^2}{t^2}\Bigg(\alpha q(t)+\frac{\beta}{q(t)}\Bigg)+\gamma \frac {q(t)}{t}+\delta\frac{q(t)(q(t)+1)}{q(t)-1}.\]
It is part of the family of the so called Painlevé equations discovered by É. Picard (1889), P. Painlevé (1900, 1902), R. Fuchs (1905), and B. Gambier (1910). All these equations are characterised by the Painlevé property: the only movable singularities are poles. A direct consequence is that any meromorphic solution germ admits an analytic continuation as a meromorphic function along any loop avoiding a finite singular set.

\medskip

Painlevé equations are strictly related to the theory of isomonodromic deformations. For instance, solutions of the fifth Painlevé equation induce a one-dimensional singular analytic foliation inside the three-dimensional algebraic variety $\Conn$, the moduli space of a certain class of rank-two irregular connections on the Riemann sphere. These connections are called \textit{PV connections} and they are characterised by having two regular singularities at $0$ and $\infty$, and one irregular singularity of order two at $1$. 
Leaves of the fifth Painlevé foliation are called isomonodromic deformations: indeed, connections that lye on the same leaf share the same monodromy representation. We also refer to this foliation as the \textit{isomonodromic foliation} associated to the fifth Painlevé equation. 

\medskip


The link between the moduli space $\Conn$ and the Painlevé V equation has been deeply studied in the last decades by different mathematicians and physicists. In 1979, K. Okamoto \cite{okamoto} introduced the algebraic spaces of initial conditions for all the Painlevé equations. In the Painlevé V case, for any time parameter $t\in \C^*$, the \textit{Okamoto space} of initial conditions is obtained through a sequence of eight explicit blow-ups of $\F_2$, the second Hirzebruch surface. We denote such surface as $Ok_t^V$ as shown in the following picture. 

\begin{center}
    \includegraphics[width=7cm]{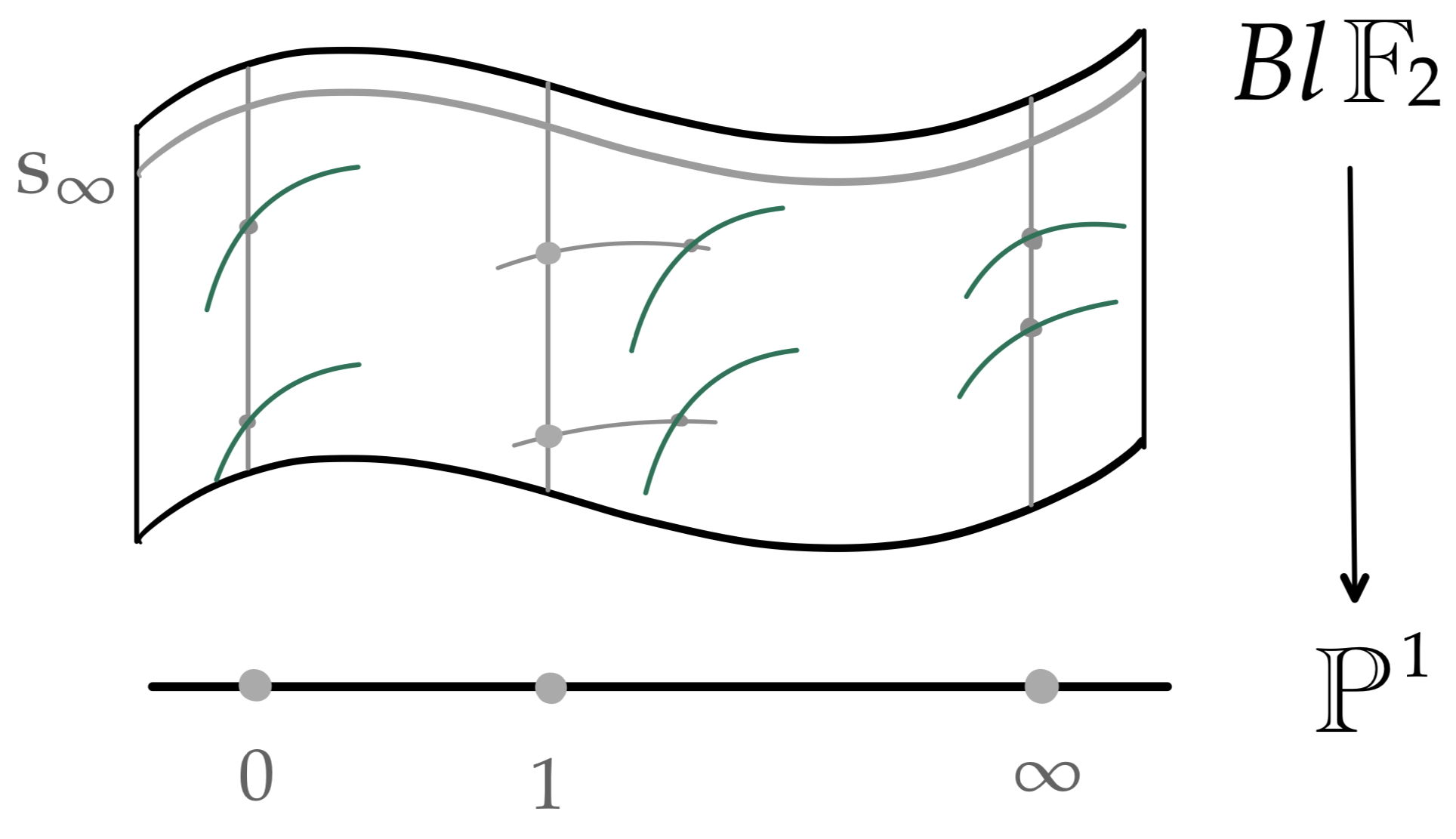}
\end{center}

In grey is represented the so called Okamoto divisor, that corresponds to some points that cannot be taken as initial condition for solutions of the Painlevé V equation: they indeed project over the poles of the connections. Note that the dependence on $t$ appears in the positions of the points in the fiber above 1.

More recently P. Boalch \cite{BOALCH2001137} studied, in an analytic framework, the moduli spaces of irregular meromorphic connection and their isomonodromic deformations, with an emphasis on symplectic structures. In 2006, M. Inaba, K. Iwasaki, and M-H. Saito \cite{inab} used Okamoto's results to study the geometry of the sixth Painlevé foliation in a completely algebraic setting, producing a first compactification of the moduli space of connections considering Okamoto spaces in family. Their moduli space comes with a fibration  $\mathfrak{Con}^{VI}_\Theta\xrightarrow{t} \C\setminus\{0,1\}$, whose fibers are the Okamoto spaces $Ok^{VI}_t$, and they proved in this geometric framework that the isomonodromic foliation is transverse to the fibration. In 2015, H. Chiba \cite{Chiba} proposed a new compactification of the Painlevé foliations using weighted projective spaces. His recent work do not explore what happens to the foliation for the time parameter $t$ equal to $0$ or $\infty$, exactly as for \cite{inab}.\\
In \cite{Matt} the moduli space of PV connections $\Conn$ is constructed and a fibration $\Conn\xrightarrow{t} \C^*$ is presented: fibers correspond to the Okamoto spaces. In the same paper, the compactification $\overline \Conn$ is introduced. It allows to extend the fibration to $\pi\colon\overline\Conn\xrightarrow{t} \P^1$. The boundary components $\pi^{-1}(0)$ and $\pi^{-1}(\infty)$ are added and their geometry is described in Theorem 4.25.

\begin{center}
    \includegraphics[width=0.5\linewidth]{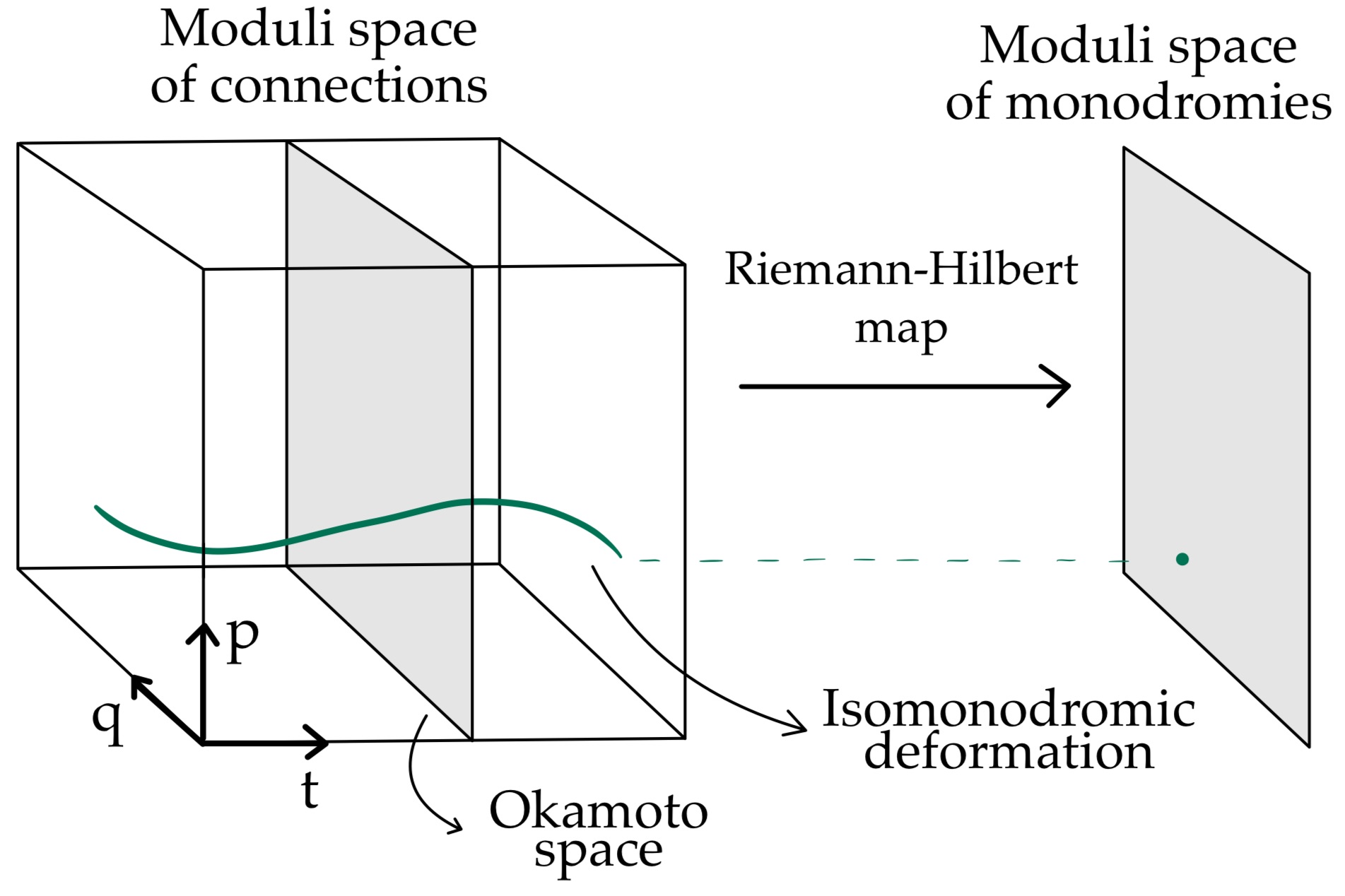}
\end{center}

The goal of this work is to study the isomonodromic foliation inside the compactification $\overline\Conn$, understanding the asymptotic behaviour of the leaves in $\pi^{-1}(0)$ and $\pi^{-1}(\infty)$, extending the compactification of the fifth Painlevé foliation. To do that, we study the local Hamiltonian vector field, as it appears in the work of Y. Ohyama \cite{Ohyama_2006}, whose integral curves are precisely the solutions of the fifth Painlevé equation defined in a dense open set of $\Conn$. Since the geometric description of $\overline\Conn$ comes with a precise system of coordinates, we are able to explicitly express the extension of the vector field to the whole compactification, leading to the following results.

\begin{theoremA}\label{thm:D}
	Let us denote by $\mathcal F^V$ the isomonodromic foliation induced by the fifth Painlevé equation. The global Okamoto divisor is invariant for $\mathcal F^V$, and the function $t$ is a first integral for the restricted foliation. Moreover, the foliation $\mathcal F^V$ is transverse to the fibration $\Conn\xrightarrow{t} \C^*$ and makes it into a local system
	\[(\Conn,\mathcal F^V)\xrightarrow{t} \C^*.\]
\end{theoremA}
In other words, inside $\overline\Conn$ we find the Okamoto-Painlevé pairs relative to the fifth Painlevé equation, and we show that the Okamoto divisor is invariant for the isomonodromic foliation. Let us show clearly that $\mathcal F^V$ is transverse to the fibration, but we do not show again (as Okamoto already did) that it is a local system. Moreover, since the Okamoto spaces coincide with the spaces of initial conditions for the fifth Painlevé equation, Theorem \ref{thm:D} is a geometric counterpart of the Painlevé property.

The second goal of this work is to describe the foliation $\mathcal F^V$ on the boundary components $\pi^{-1}(0)$ and $\pi^{-1}(\infty)$ inside $\overline\Conn$. We recall, as proven in \cite{Matt}, that $\pi^{-1}(0)$ has two irreducible components both isomorphic to $\F_1$, while $\pi^{-1}(\infty)$ has three irreducible components: two of them are isomorphic to $\F_1$, and the one in the middle is isomorphic to $\F_0$.
\begin{theoremA}\label{thm:penc}
	The boundary components $\pi^{-1}(0)$ and $\pi^{-1}(\infty)$ are invariant for $\mathcal F^V$. The foliation restricted to each $\F_1$ component of $\pi^{-1}(0)$ and $\pi^{-1}(\infty)$ is a pencil of conics, while restricted to the $\F_0$ component of $\pi^{-1}(\infty)$ it is a pencil of genus-one curves.
\end{theoremA}

To prove Theorems \ref{thm:D} and \ref{thm:penc} we need to find first integrals of the restricted foliation. To do that, we should understand how the limit connections lying on these boundary components are made, and deduce some monodromy invariants. In \cite{Matt} the following result is proven.
\begin{lem*}
    The connections lying on the boundary components of $\overline\Conn$ are connections defined on irregular stable nodal curves (as in Figure below). They correspond to the data of two distinct connections: a (possibly confluent) Heun connection on the upper smooth component and a (possibly confluent) hypergeometric connection on the lower. Moreover, the node of the curve is a pole for both connections.
\end{lem*}
\begin{center}
    \includegraphics[width=9cm]{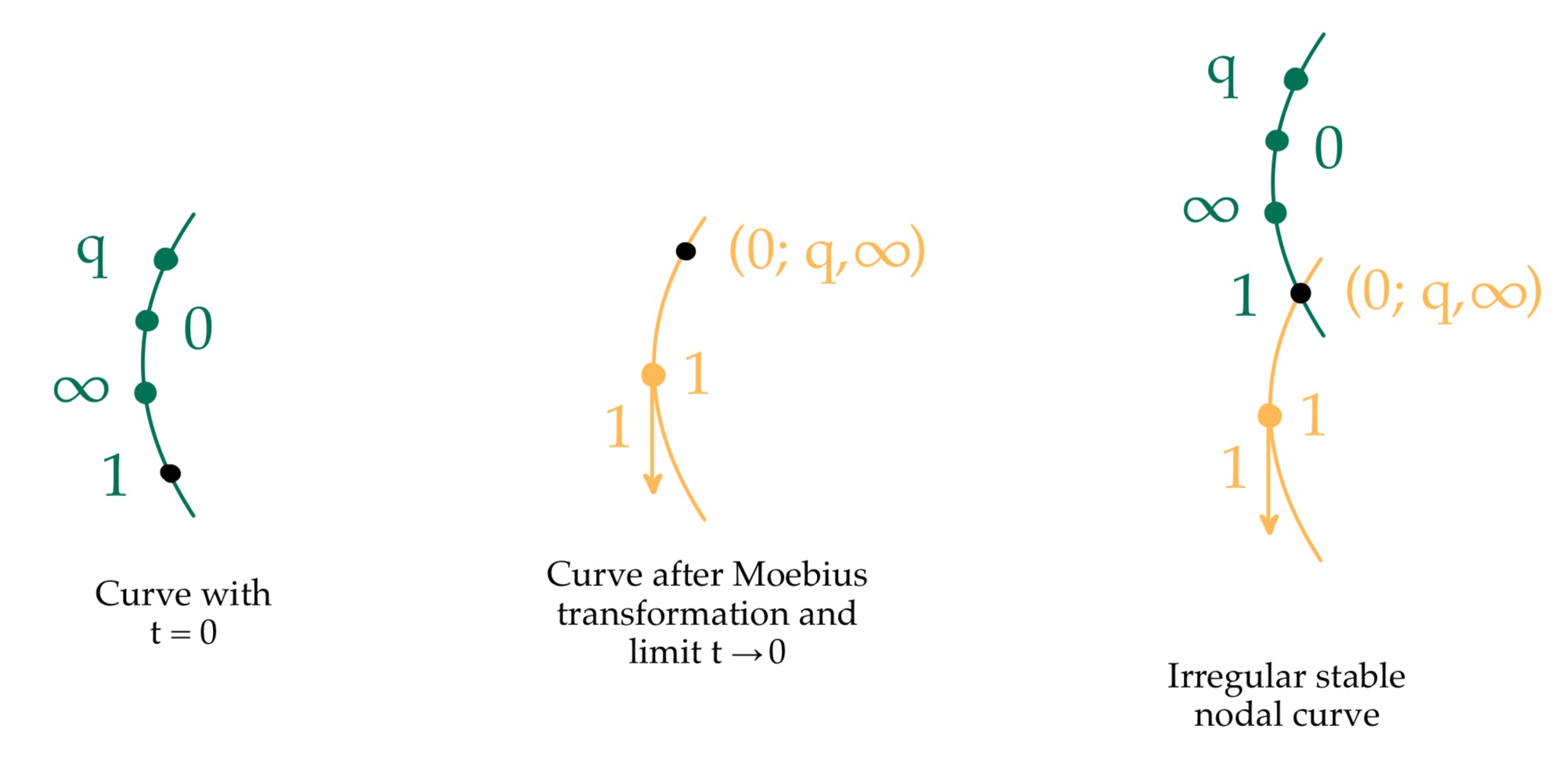}
\end{center}
We then deduce that, since we want the monodromy to remain untouched during the deformation, the residual spectral datum in the node must be the same for any connection lying on a isomonodromic leaf. We then prove the following.
\begin{lem*}
    The residual spectral data at the nodes of the irregular stable nodal curves are first integrals of the hamiltonian vector field restricted to the corresponding boundary component.
\end{lem*}
Thanks to this lemma we are able to  describe the integral curves and understand the isomonodromic deformations along all the boundary components.

\vfill

\textit{Acknowledgements:}\\
This work is part of my Phd project, financed by the France 2030 program, Centre Henri Lebesgue ANR-11-LABX-0020-01, and Rennes university.

I would like to thank my advisor Frank Loray for guiding me through this beautiful journey, Jorge Vitorio Pereira for the rich discussions about birational geometry of foliations and Matilde Maccan for her mathematical and emotional support.

I also received a financial support form Collège doctorale de Bretagne, Rennes Metropole, EUR caps, IRIS-E, and the RFBM program for my mobility in Rio de Janeiro.

\newpage

\section{PV Connections}
The contents of this section can be be found in details in Section 1 of \cite{Matt}. To fix notations and make the discussion more clear, we repropose here the main definitions. In the end, we describe the irregular monodromy.
\begin{defn}
   A linear meromorphic connection over $\P^1$ is a triple $(E,D, \nabla)$, where 
    \begin{itemize}
    \item $E\to \P^1$ is a holomorphic vector bundle, 
        \item $D$ is an effective divisor of $\P^1$,
        \item $\nabla$ is a $\C$-linear application $\nabla\colon \mathcal E \to \mathcal E\otimes \Omega^1_{\P^1}(D),$
    satisfying the Leibniz rule given by the $\O_X$-module structure of the sheaf $\mathcal E$ of holomorphic sections: $\nabla (f\cdot \sigma)=df\cdot\sigma+f\cdot\nabla\sigma$.
    \end{itemize} 
    The divisor $D$ is called \emph{polar divisor} and it describes the order and the position of the poles that the connection must have. We recall the notation $\Omega^1_{\P^1}(D):=\Omega^1_{\P^1}\otimes\O_{\P^1}(D)$. From now on, we will omit the $\P^1$ in the sheaves notation.
\begin{defn}
    A \textit{meromorphic gauge transformation} $\Psi\colon(E,D, \nabla)\to (E,D, \nabla)$ is a birational morphism of bundles $\Psi\colon E\to E'$ such that $\Psi^*\nabla'=\nabla$.
\end{defn}
We will say that two linear meromorphic connections are \textit{meromorphically gauge equivalent}, or just \textit{equivalent}, if there exists a meromorphic gauge transformation between them.
\end{defn}
\begin{defn}
    A point $a\in \P^1$ is called a \textit{singularity} (or a \textit{pole}) for the connection $(E,D,\nabla)$  if $a\in |D|$. Its \textit{Poincaré rank} is the integer number $\deg (D_{|a})-1$. 
\end{defn}
\begin{defn}
    A singularity is said \textit{logarithmic} if its Poincaré rank is zero. Otherwise, it is said \textit{regular} if it can be reduced via a (meromorphic) gauge transformation to a logarithmic singularity and \textit{irregular} if not.
\end{defn}
Meromorphic connections have an explicit local description: let $U\subseteq \P^1$ be a trivialising open set for the vector bundle $E$, then the connection here express as 
\[\nabla_U=d+\Omega_U\;\;\;\text{ for }\;\;\;\Omega_U\in \mathfrak{gl}_2(\Omega^1(D_{|U})).\]
The matrix $\Omega_U$ is called the connection matrix relative to the trivialising open set $U$. Note that it also depends on the chosen trivialisation. 

If $D_{|U}=n[a]$, then the connection matrix express
\[\Omega_U=\sum_{k=1}^n A_k^{(a)}\frac{dx}{(x-a)^k} +\text{holomorphic terms}\]
for $A_k^{(a)}\in \mathfrak{gl}_2(\C)$.
\begin{defn}
We call $\sum_{k=1}^n A_k^{(a)}\frac{dx}{(x-a)^k}$ the \textit{principal part} of $(E,D, \nabla)$ around the singularity $a\in \P^1$. Moreover, we call $A_1^{(a)}=\mathrm{Res}_{x=a}\Omega_U$ the \textit{residual matrix} of the connection $(\nabla,E,D)$ at $x=a$. 
\begin{oss}
    If $a$ is a logarithmic singularity, then the spectrum of $A_1^{(a)}$ is invariant under meromorphic gauge transformation, otherwise it is not. See for instance Example 2 in \cite{Matt}.
\end{oss}
\end{defn}
We are interested in giving a more precise definition of some residual data, invariant under the gauge action.
\begin{defn}
    The \textit{residual spectral datum} of a connection $(E, D, \nabla)$ at a pole $x=a$ is computed as follows. Let $\Omega$ be any connection matrix for $(E, D, \nabla)$ in an open trivialising neighbourhood of $a$, and let $\Omega^{<0}$ be its principal part. Up to a meromorphic gauge transformation, we can diagonalise it. Let $D_1^{(a)}$ the new (diagonal) residual matrix and denote by $\{\kappa_a^+, \kappa_a^-\}$ its spectrum. It does not depend on the meromorphic gauge transformation we have chosen. We refer to the \textit{residual spectral data} of a connection $(E, D, \nabla)$ as the collection 
    \[\Big(\{\kappa_a^+,\kappa_a^-\}\Big)_{a\in|D|}.\]
\end{defn}
We specialise our investigation to a particular class of meromorphic connections over $\P^1$.
\begin{defn}
     A meromorphic connection $(E,D, \nabla)$ over $\P^1$ is of \textit{Painlevé type} if $\mathrm{rank} E=2$ and the minimal polar divisor within its meromorphic gauge equivalence class has degree four.
\end{defn}
\begin{oss}
    A complete classification of Painlevé type connections can be found in Section 1.4 of \cite{Ohyama_2006}. Up to Moebius transformations, the admissible Painlevé type divisors are then the following
    \begin{multicols}{2}
 \begin{itemize}
    \item $D^{VI}:=[0]+[1]+[t]+[\infty]$,
    \item $D^{V}:=[0]+2[1]+[\infty]$,
    \item $D^{IV}:=[0]+3[1]$,
\end{itemize}
\columnbreak
\begin{itemize}
    \item $D^{III}:=2[0]+2[1]$,
    \item $D^{II}:=4[\infty]$.
\end{itemize}
\end{multicols}
\end{oss}
We give now one of the main definitions of this work.
\begin{defn}
    A Painlevé type connection is of \textit{PV type} if the minimal divisor within its meromorphic gauge equivalence class is equivalent to $D^V$.
\end{defn}
Painlevé type connections with fixed residual spectral data are the first example of meromorphic connections over $\P^1$ with a positive dimensional moduli space (w.r.t. meromorphic gauge equivalence). Indeed, the dimension of their moduli space depends on the degree of the polar divisor and on the rank of the vector bundle. If we consider rank 2 connections, we need at least a polar divisor of degree four to have a moduli space of positive dimension. When the polar divisor has degree three, we have the so called hypergeometric systems. Their moduli space, with fixed residual spectral data up to integer shifts, corresponds to a point (see Section 1.3.5 of \cite{geompanv}).

\subsection{Normal Form}\label{secnormform}
In a recent paper, Diarra and Loray \cite{Diarra} proved that, up to meromorphic gauge transformation, any PV connection is equivalent to a connection $(\nabla, \O\oplus \O(2), D^V+[q])$ with the following connection matrix in $U_0:=\P^1\setminus\{\infty\}$
\[\Omega_0=\begin{pmatrix}0&1\\0&t\end{pmatrix}\frac{dx}{(x-1)^2}+\begin{pmatrix}0&-1\\0&-\kappa_1\end{pmatrix}\frac{dx}{x-1}+\begin{pmatrix}0&1\\0&-\kappa_0\end{pmatrix}\frac{dx}{x}+\begin{pmatrix}0&0\\-\rho^{V}&0\end{pmatrix}xdx+\begin{pmatrix}0&0\\\hat p&-1\end{pmatrix}\frac{dx}{x-q}+\begin{pmatrix}0&0\\\hat K&0\end{pmatrix}dx,\]
where 
\begin{itemize}
    \item $\kappa_0, \kappa_1, \kappa_\infty$ are generic residual spectral data;
    \item $\rho^V=\frac{(\kappa_0+\kappa_1-1)^2}{4}-\frac{\kappa_\infty^2}{4}$;
    \item $\hat K=\frac{\hat p^{2}}{q \left(q -1\right)^{2}}+\frac{\left(\mathit{\kappa_0} \left(q -1\right)^{2}+(\kappa_1-1)  q \left(q -1\right)-t q \right) \hat p}{q \left(q -1\right)^{2}}+\rho^V q +\frac{\hat p}{q -1}$;
    \item $t, \hat p$ and $q$ are three free parameters uniquely defining the connection.
\end{itemize}
\begin{oss}
        In the literature, as in \cite{Ohyama_2006}, the parameter $\kappa_1$ can be found under the form as $\theta+1$. The advantage of using $\theta$ is its more direct link to the Painlevé fifth equation \ref{secpaineq}.
\end{oss}
\begin{defn}
   A singularity for a meromorphic connection is said \textit{apparent} if it can be removed via a (meromorphic) gauge transformation. They have nilpotent or diagonalizable residual matrix with integer eigenvalues and trivial monodromy.
\end{defn}
\begin{lemma}
    The singularity $x=q$ in the normal form above is apparent.
\end{lemma}
\begin{proof}
    Let us consider the chart $U_0\subseteq \P^1$. If we apply the meromorphic gauge transformation
    \[\Psi(x)=\begin{pmatrix}1&0\\\hat p&x-q\end{pmatrix}\]
    the singularity disappears.
\end{proof}
It has been shown (\cite{Diarra}, \cite{Matt}) that a generic class of PV connections is uniquely represented by the three parameters $t,q,\hat p$ appearing in the normal form. 
\begin{oss}
    Let $f\in \mathrm{Aut}(\P^1)$ be a Moebius transformation. It acts trivially on the moduli space of connections via pull-back, since $(E,D, \nabla)$ and $f^*(E,D, \nabla)$ are equivalent. In particular $f^*$ acts the following way on the three parameters
    \[f^*\colon (t,q,\hat p)\mapsto (Df(1)\cdot t, f(q), \hat p).\]
    Notice that $f^*(E,D, \nabla)$ is no more in normal form since the poles have been moved.
\end{oss}

\subsection{Rational Irregular Curves and their Nodal Version} \label{secirrcurv}
To describe the moduli space of PV connections, we need the following definition, that has been introduced in \cite{irrcurv} and \cite{boalchirrcurv}. As we will precise in Section \ref{secmodsp}, the datum of a class of PV connections is equivalent to the data of a rational regular curve $(\P^1, D^V+[q])$ and an extra complex parameter $\hat p$. 
\begin{defn}
    An irregular curve $(\mathcal C, D, J)$ is the datum of a complex curve $\mathcal C$, an effective divisor $D=\sum_{i\in I} a_i[p_i]\in \mathrm{Div}(\mathcal C)$ and a collection of jets of coordinates $J=\big\{j^{a_i-1}\gamma_i\big\}_{i\in I}$, where $\gamma_i\colon \R_0\to \P^1$ is a smooth curve such that $\gamma_i(0)=p_i$. Sometimes we will omit $D$ or $J$ in the notation.
\end{defn}
\begin{es}
    We can see $\P^1$ with five punctures as a rational irregular curve $(\P^1 ,[a]+[b]+[c]+[d]+[e])$.
\end{es}
\begin{oss}
     An automorphisms $\phi\in\mathrm{Aut}(\P^1)$ acts on $(\P^1,D, J)$ as $(\phi(\P^1), \phi_*D, \phi_*J)$. If we call $p=\gamma(0)$, we have that $\phi_*\big(j^k\gamma\big)=\big(\phi(p), D\phi(p)\left(\gamma^{(1)}(0)\right),\dots,D^k\phi(p)\left(\gamma^{(k)}(0)\right)\big)$.
\end{oss}
In \cite{Matt} a compactification, inspired by the Deligne and Mumford's technics, of the moduli space of rational irregular curves $(\P^1, [a]+2[b]+[c]+[d])$ has been constructed. Up to Moebius transformation such irregular curves are equivalent to $(\P^1, D^V+[q])$. The different jets are
\[j^0(0)=0, \;\;\;\,j^1(1)=(1,t), \;\;\;\,j^0(\infty)=\infty,\;\;\;\, j^0(q)=q, \]
for $q\in \P^1\setminus\{0,1,\infty\}$ and $t\in T_1\P^1\setminus\{0\}$.
In particular the two free parameters are $t\in T_1\P^1$, a tangent vector of $\P^1$ at $x=1$, and $q\in \P^1$, a point. For any admissible value of $t$ and $q$, we get a different class of rational irregular curves, getting a moduli space
\[\mathcal M:=\Big\{q\in\P^1\setminus\{0,1,\infty\};\;\;t\in T_1\P^1\setminus \{0\}\Big\}\cong\P^1\times\P^1\setminus(Q^0\cup Q^1\cup Q^\infty\cup A_0\cup A_\infty),\]
where $Q^i:=\{q=i\}$ and $A_j=\{t=j\},$ as shown in the picture below.
\begin{center}
    \includegraphics[width= 9cm]{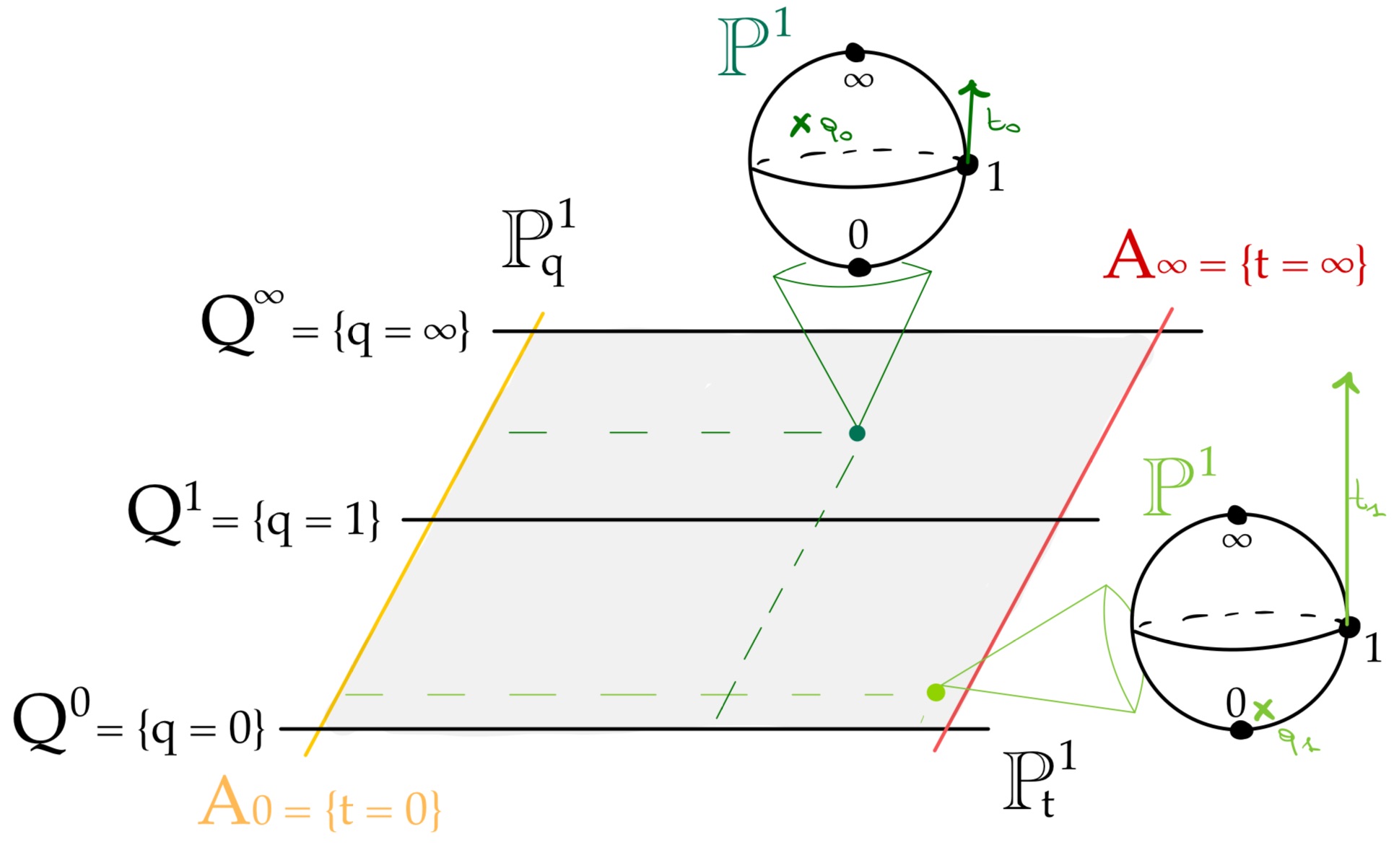}
    \captionof{figure}{}\label{pic1}
\end{center}
To compactify this space, we need to understand what happens for the critical values of $t$ and $q$. In \cite{Matt} we define the limit objects as textit{irregular stable nodal curves}, and,
in Theorem 4.5, we show that the compactification of $\M$ is given by the contraction of the $[-2]$-curve obtained by blowing-up $\P^1\times\P^1$ three times, as shown in the picture below:
\begin{center}
    \includegraphics[width = 10 cm]{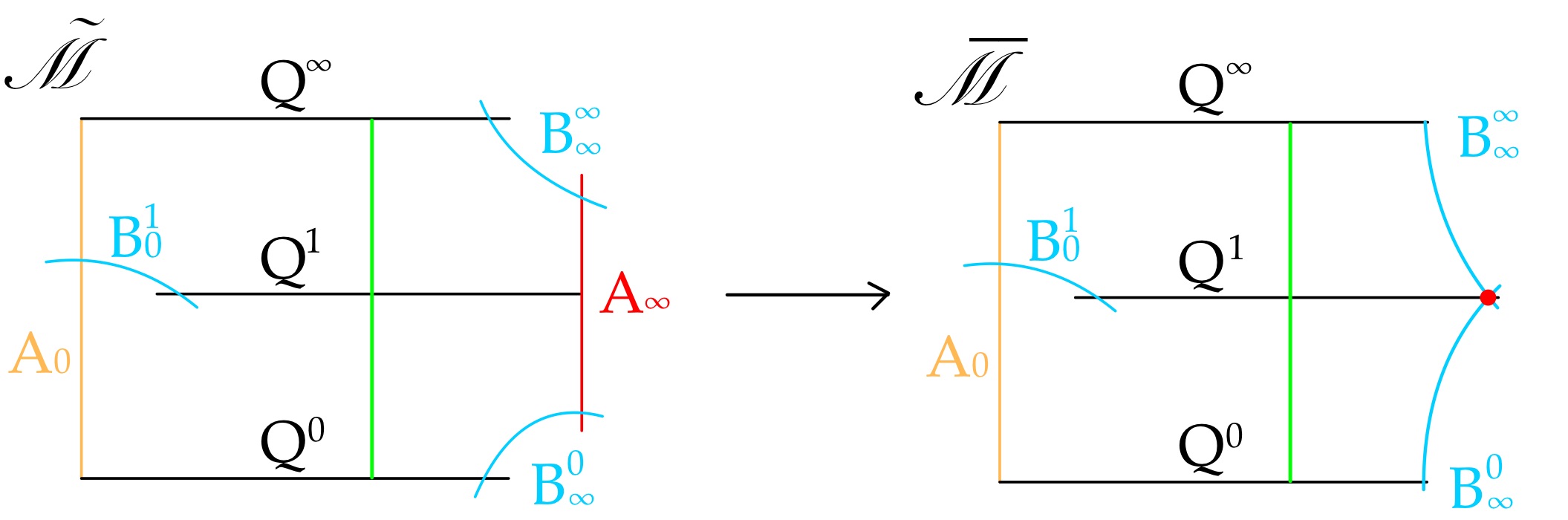}
    \captionof{figure}{}\label{modspcurves}
\end{center}
A precise description of the universal curve and of the irregular stable nodal curves arising can be found in \cite{Matt} Section 4.2. 
\begin{es}
    The boundary $A_0$ represents those curves with $t=0$, but it turns out that setting $t=0$ is not the only possible representation. We recall indeed that we are considering curves up to isomorphism and, in particular, we can apply a Moebius transformation $f_{A_0}$ that fixes 1 and such that $Df_{A_0}(1)\cdot t= 1$. We can then study the limit for $t\to0$ and look at what the resulting curve looks like. One such Moebius transformation is 
\[
	f_{A_0}(x)=\frac{(x+1)t}{(t-2)x+t+2}.
\]
It holds that 
\[\lim_{t\to0}f_{A_0}(0)=\lim_{t\to0}f_{A_0}(q)=\lim_{t\to0}f_{A_0}(\infty)=0.\]
Up to Moebius transformations, we have then two different representations of the limit $t\to 0$. We can indeed represent the limit rational irregular curve in such a way that it keep trace of both the possibility, giving rise to a stable nodal curve in a Deligne-Mumford flavour.
	\begin{center}
		\includegraphics[width=.6\linewidth]{image1-3.jpeg}
        \captionof{figure}{}\label{fig:A0}
	\end{center}	
\end{es}

\subsection{Moduli Space}\label{secmodsp}
As we have seen in Section \ref{secnormform}, a generic class of a PV connection is uniquely determined by the three parameters $(t,q,\hat p)$ that appear in its normal form. In Section \ref{secirrcurv} we have given a geometrical interpretation to the parameters $(t,q)$ as respectively a tangent vector of $\P^1$ in $x=1$, and a point of $\P^1$.

Until $t\neq0$ and $q\neq0,1,\infty$, the moduli space of PV connection is the trivial line bundle $\M\times \C$ with coordinates $(t,q,\hat p)$. We extended in \cite{Matt} the line bundle over $\widetilde\M$, getting 
\[\O_{\widetilde \M}(D)\;\;\;\;\text{ for }\;\;\;\;D= A_0+2Q^1+2B_0^0-B_\infty^0-B_\infty^\infty,\]
that turns out to be trivial over all the $Q^i$ and $A_\infty$. We can finally compactify it by adding the section at infinity.

The connections lying over the boundary components of $\widetilde \M$ are indeed connections over stable nodal irrational curves. In all the cases we get two smooth components and a (confluent) hypergeometric connection on one of them, and a (confluent) Heun connection on the other.

Studying the explicit limits of the connections we proved that to get the compactification we needed to blow-up some special trivialising section over the $Q^i$ and $A_\infty$. We do not consider the complementary of the exceptional divisor inside the $Q^i$ and $A_\infty$. Let us call $\overline \Conn$ the resulting 3-fold, that we proved being the compactification of the moduli space of PV connection we are looking for. This compactification comes with a fibration $\pi\colon\overline\Conn\xrightarrow{t} \P^1$ whose fibers have been described in Theorem 4.25 of \cite{Matt}. For any value of $t_0\in \C^*$, we denote by
\[Ok_{t_0}:=\pi^{-1}(t_0)\]
the Okamoto space \cite{okamoto}, that is the space of the initial conditions for the fifth Painlevé equation \cite{JIMBO1981306}. See, for instance, Section 2 of \cite{Matt} for more details about its geometry. 
\begin{center}
    \includegraphics[width=12cm]{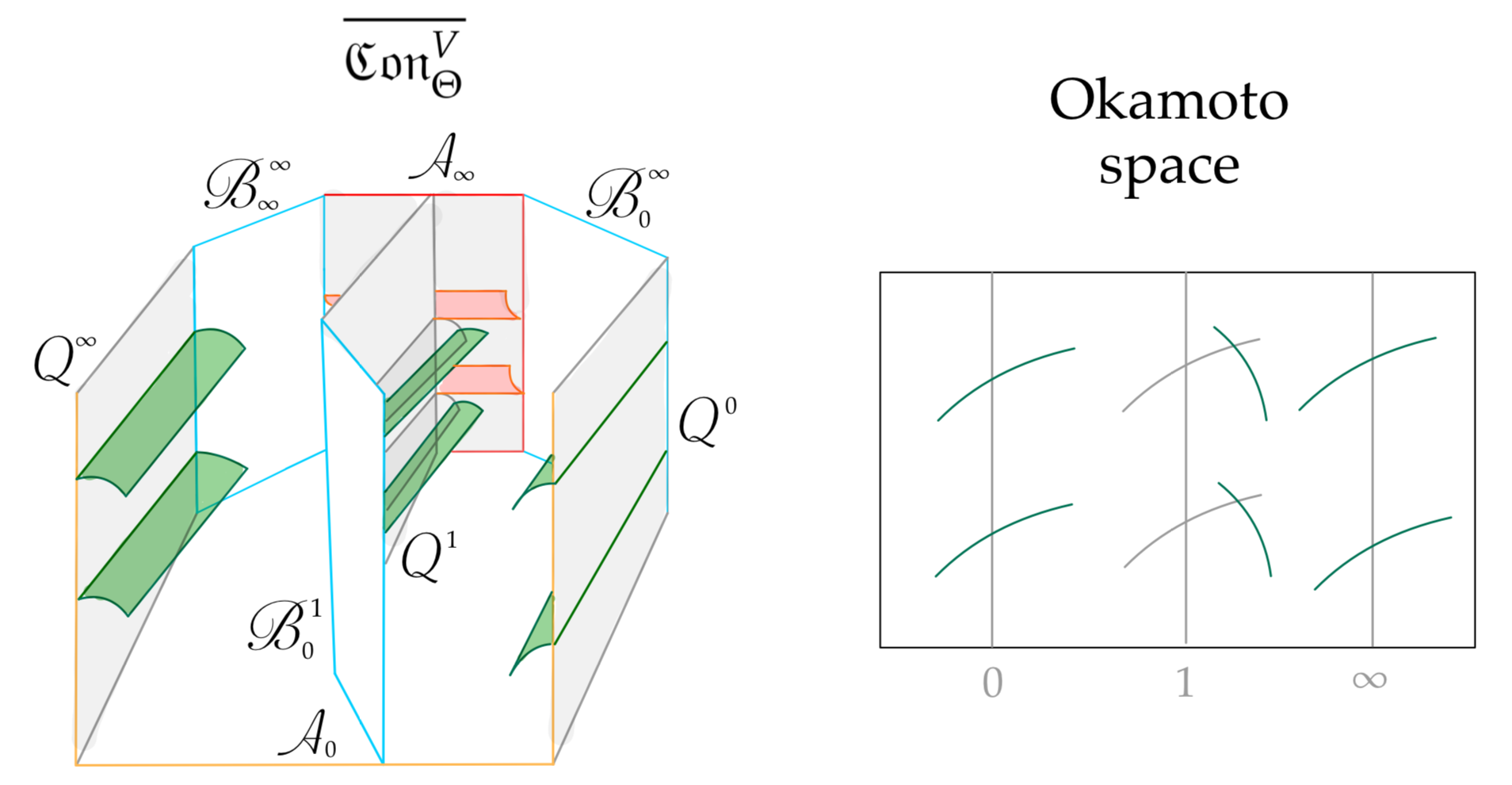}
     \captionof{figure}{}\label{figmodsp}
\end{center}
We add some notation for the exceptional divisors: we will call $\mathcal E^i_\pm$ the exceptional divisors arising from the blowing up of the sections in $\mathcal Q^i$. 

We conclude this section by recalling that a connection in $\M\times\C\subseteq\overline\Conn$ is the datum of $(t,q,\hat p)$ that is an irregular curve $(\P^1, D^V+[q])$ and an additional complex parameter $\hat p$. On the other hand, a class of PV connections lying in a boundary component of $\overline \Conn$ is the datum of an irregular stable nodal curve and of an extra complex parameter. A connection on such irregular stable nodal curves decomposes as a connection on each smooth component. In our case, we will always have a (confluent) hypergeometric connection on one component and a (confluent) Heun connection on the other. The node of the curve is a pole for both connections, and its monodromy will play a central role.

\subsection{Regular and Irregular Monodromy}
The monodromy of a meromorphic connection measures the evolution of a (local) horizontal section, that is a section $Y$ satisfying $\nabla Y=0$, under analytic continuation along a closed path around singularities (indeed, being $\P^1$ simply connected, no monodromy data is given by the topology). 

The PV connections present three kinds of singularities: an apparent singularity in $x=q$, two logarithmic poles at $x=0,\infty$ and a double pole at $x=1$. Understanding the global monodromy representation of such a connection is very complicated, but we will give a full description of the local monodromies.

Let us start by the logarithmic case.\\
We can consider an open trivialising set $U\subseteq \P^1$ containing only the logarithmic singularity $0$ (or $\infty$), and a smooth loop $\gamma\colon \
\mathbb{S}^1\to \P^1$ entirely contained in $U$ and encircling the singularity. Let us set $Y^\gamma$ the analytic continuation of the horizontal section $Y$ along $\gamma$. It holds that, for each $x\in U$
\[Y^\gamma(x)=Y(x)\cdot M^\gamma\]
for a constant matrix $M^\gamma\in \mathrm{GL}_2(\C)$ that is called the monodromy matrix.
It is a classical result (Lemma 1.1.4 \cite{RHrank2}) that $M^\gamma$ only depends on the homotopy class $[\gamma]\in \pi_1(U\setminus\{a\})$, for $a=0,\infty$. 
We deduce hence that nearby a logarithmic singularity we have a local monodromy representation 
\begin{align*}
    \rho_a\colon\pi_1(U\setminus\{a\})&\to \mathrm{GL}_2(\C)\\ [\gamma]&\mapsto M^\gamma. 
\end{align*}
If we consider a loop $\gamma$ turning one times around 0, the monodromy matrix $M^\gamma$ is conjugated to the exponential of the residual matrix $A^{(0)}_{1}$.\\
Notice that an apparent singularity is detected as a logarithmic pole, but its local monodromy is trivial. In fact, $\exp(A^{(q)}_1)=I$.

When the singularity is irregular, as the pole of order 2 in $x=1$, the description of the local monodromy is more complicated. Since a large literature can be found on irregular monodromy (for instance, \cite{RHrank2}, \cite{BOLIBRUCH2006235} or \cite{boalchirrcurv}), let us concentrate on the case of the irregular singularity of a PV connection in its normal form as in Section \ref{secnormform}. Let $U\cong \D$ be a trivialising open set for the PV connection $(E, D^V, \nabla)$ containing only the irregular singularity at $x=1$. Compared to the logarithmic case, parallel sections in $\D$ have only a formal description on the entire open punctured set $\D^*:=\D\setminus\{1\}$, and these formal series converges into holomorphic sections only on the two sectors
\[V^k:=\Big\{x\in\D^*\;\;\Big|\;\;\arg(x)\neq\arg (t)+(-1)^k\frac\pi2\Big\}\;\;\; \text{ for }\;\; k=0,1,\]
where $t$ is the non zero eigenvalue of the second order matrix at the irregular pole. It is, equivalently, the complex number defining the jet $j^1\gamma_1=(1,t)$ in the associated rational irregular curve.
\begin{oss}\label{rem:onedim}
	In order to understand what is happening, let us look at the one dimensional case. Consider the equation 
	\[dy=a(x)\frac{dx}{x^2}y,\;\; \text{with}\;\; a(x)=t+a_1x+a_2x^2+\dots\]
	A solution is given by the explicit formula
	\[y(x)=\exp\int\frac{a(x)}{x^2}dx=\exp\Bigg(-\frac{t+o(1)}{x}\Bigg).\]
	The condition $\mathrm{Re}(t/x)=0$ consists of 2 rays, that we will call \textit{anti-Stokes directions}, dividing the disk into two equal sectors, that are precisely the $V^k$ as before. An easy computation indeed shows that
		\[\mathrm{Re}\Bigg(\frac tx\Bigg)=\frac{\rho_t}{\rho_x}\mathrm{Re}\Bigg(\frac{e^{i\theta_t}}{e^{i\theta_x}}\Bigg)=\frac{\rho_t}{\rho_x}\mathrm{Re}\Big(e^{i(\theta_t-\theta_x)}\Big)=0 \iff \theta_t-\theta_x=\frac\pi2+k\pi, \;k\in \Z.\]
		Note that such rays corresponds to the sign change for Re$(t/x)$. Consider now the bisectors rays of sectors $V^0\cap V^1$ and let us call them \textit{Stokes directions}. They corresponds to the solutions of the equation $\theta_t-\theta_x=k\pi$, and hence Re$(t/x)=\pm\rho_t/\rho_x$. In particular on one Stokes direction the solution $y(x)$ tends to infinity, while, on the other, to zero.
			\begin{center}
				\includegraphics[width=0.7\linewidth]{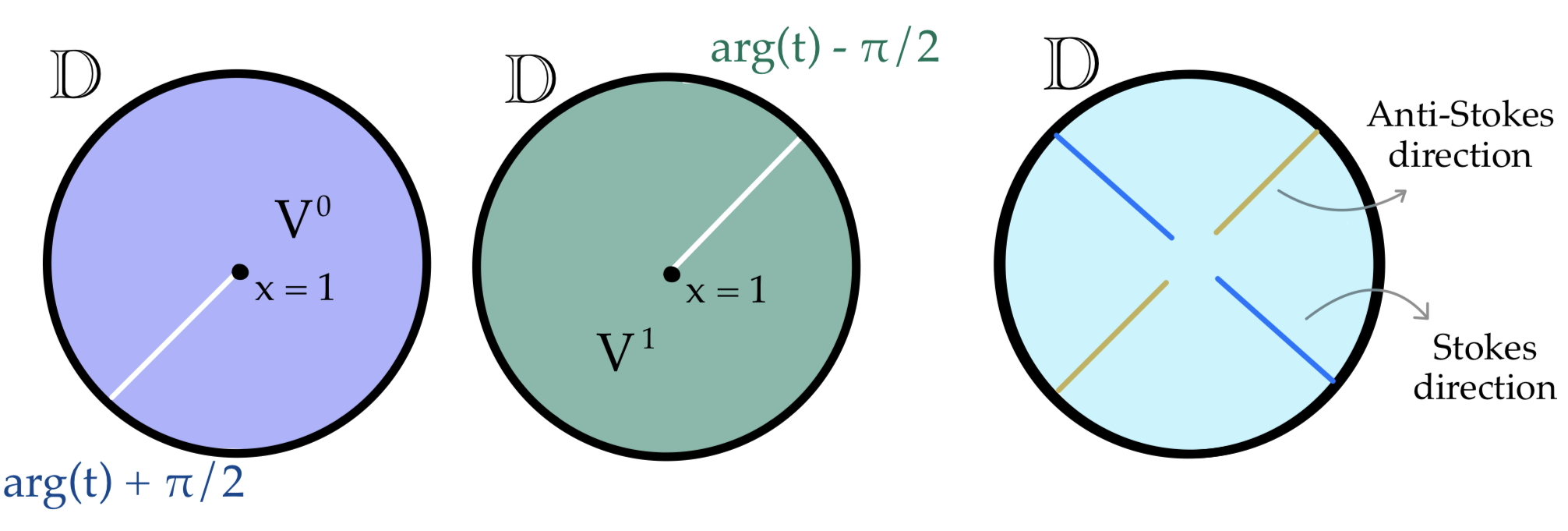}
			\end{center}
\end{oss}

The previous remark explains why solutions cannot converge on the whole punctured disk. 

Irregular monodromy is also more complicated: passing to a sector from another also generates monodromy. Details on irregular monodromy and Stokes phenomena are beyond this thesis end we refer to paragraph 21.1 of \cite{Ily} and section 1.3 of \cite{RHrank2} for a detailed description. 

Let us then jump to two fundamental results that describe local irregular monodromy. 
\begin{lemma}
	Let $(E,2[0], \nabla)$ an irregular connection on the disk. Let the connection matrix $\Omega$ be such that its principal part is diagonal. Consider $\{\kappa^+, \kappa^-\}$ the residual spectral data. The monodromy relative to a loop $\gamma$ encircling the singularity $x=0$ is
	\[M^\gamma:=S_0\cdot S_1\cdot D=\begin{pmatrix}1&c_0\\0&1\end{pmatrix}\begin{pmatrix}1&0\\c_1&1\end{pmatrix}\begin{pmatrix}e^{2i\pi\kappa^+}&0\\0&e^{2i\pi\kappa^-}\end{pmatrix}\]
	for some complex numbers $c_0,c_1$. Matrices $S_0$ and $S_1$ are called Stokes matrices.
\end{lemma}
\begin{proof}
	Proposition 1.3.8 of \cite{RHrank2}.
\end{proof}

We are now ready to build the local irregular monodromy representation. In particular we associate to each path in the local fundamental groupoid its monodromy matrix. Let us give in the figure below a presentation to the groupoid, as in \cite{RHrank2}. 
	\begin{center}
		\includegraphics[width=0.25\linewidth]{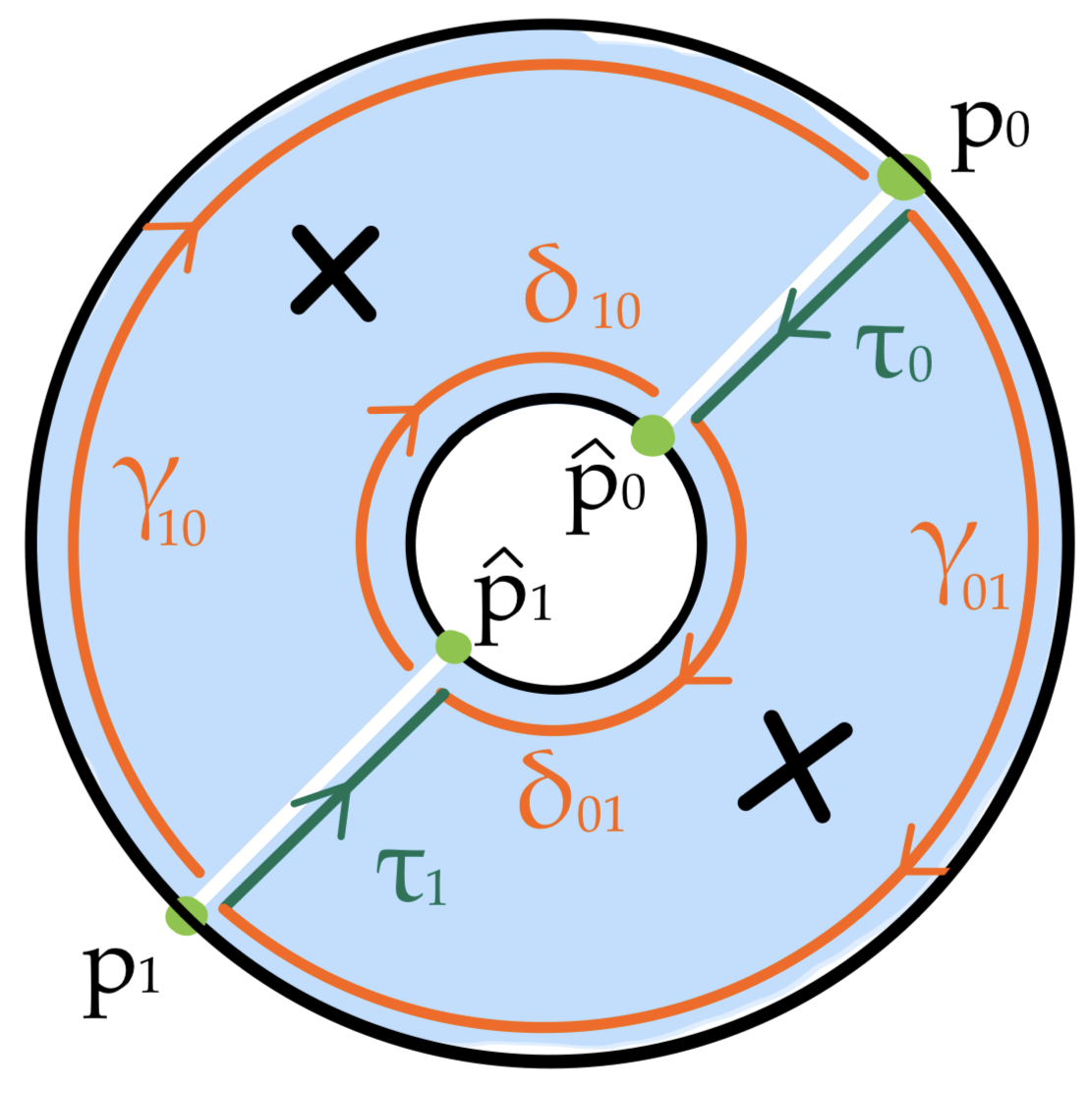}
	\end{center}

In particular, we can associate a monodromy matrix to each path. We denote by $\star$ the concatenation of loops, with the convention that $\gamma \star \delta$ means “first follow $\gamma$, then $\delta$”.
\[
\begin{matrix}
	\gamma=\gamma_{0,1}\star\tau_1\star\delta_{0,1}^{-1}\star\tau_0^{-1}&\;\;&\rightsquigarrow&\;\;&M^\gamma=\begin{pmatrix}1&c_0\\0&1\end{pmatrix}=:S_0\\\\
	\gamma=\gamma_{1,0}\star\tau_0\star\delta_{1,0}^{-1}\star\tau_1^{-1}&&\rightsquigarrow&&M^\gamma=\begin{pmatrix}1&0\\c_1&1\end{pmatrix}=:S_1\\\\
	\gamma=\delta_{1,0}\star\delta_{0,1}&&\rightsquigarrow&&M^\gamma=D\\\\
	\gamma=\gamma_{0,1}\star\gamma_{1,0}&&\rightsquigarrow&& M^\gamma=:S_0\cdot S_1\cdot D
\end{matrix}
\]
giving the local monodromy representation
\[\rho_{\D}\colon\pi_1(\hat \D\setminus\{s_0, s_1\}, \{p_0,p_1,\hat p_0, \hat p_1\})\to \mathrm{GL}_2(\C).\]

We would like to describe the global monodromy representation for a PV connection. To do so, we first need to define the global wild fundamental groupoid, that is constructed as follows in \cite{paulramis}. Consider $\P^1$ and perform a real blowing up at the points $0,1$ and $\infty$. Let us denote by $\hat \P^1$ this new curve, and by $D_0, D_1$ and $D_\infty$ the exceptional divisors. We recall that, since we are performing a real blow-up, they are all isomorphic to a circle. Remove then two points $s_1^+$ and $s_1^-$ close to $D^1$, lying on two opposite Stokes lines. Choose then four base points for the fundamental groupoid, $p_0\in D_0$, $p_\infty\in D_\infty$ and two points $p_1, p_2\in D_1$, lying on two opposite anti-Stokes lines. Finally, consider the groupoid whose morphisms are the paths between these four base points up to homotopy:
\[\pi_1^V:=\pi_1\left(\hat\P^1\setminus\{s_1^+, s_1^-\},\{p_0,p_\infty,p_1,p_2\}\right).\]
It is the \textit{wild fundamental groupoid} associated to any PV connection. In the figure below, we show a set of generators for $\pi^V_1$.\\
	\begin{center}
		\includegraphics[width=0.7\linewidth]{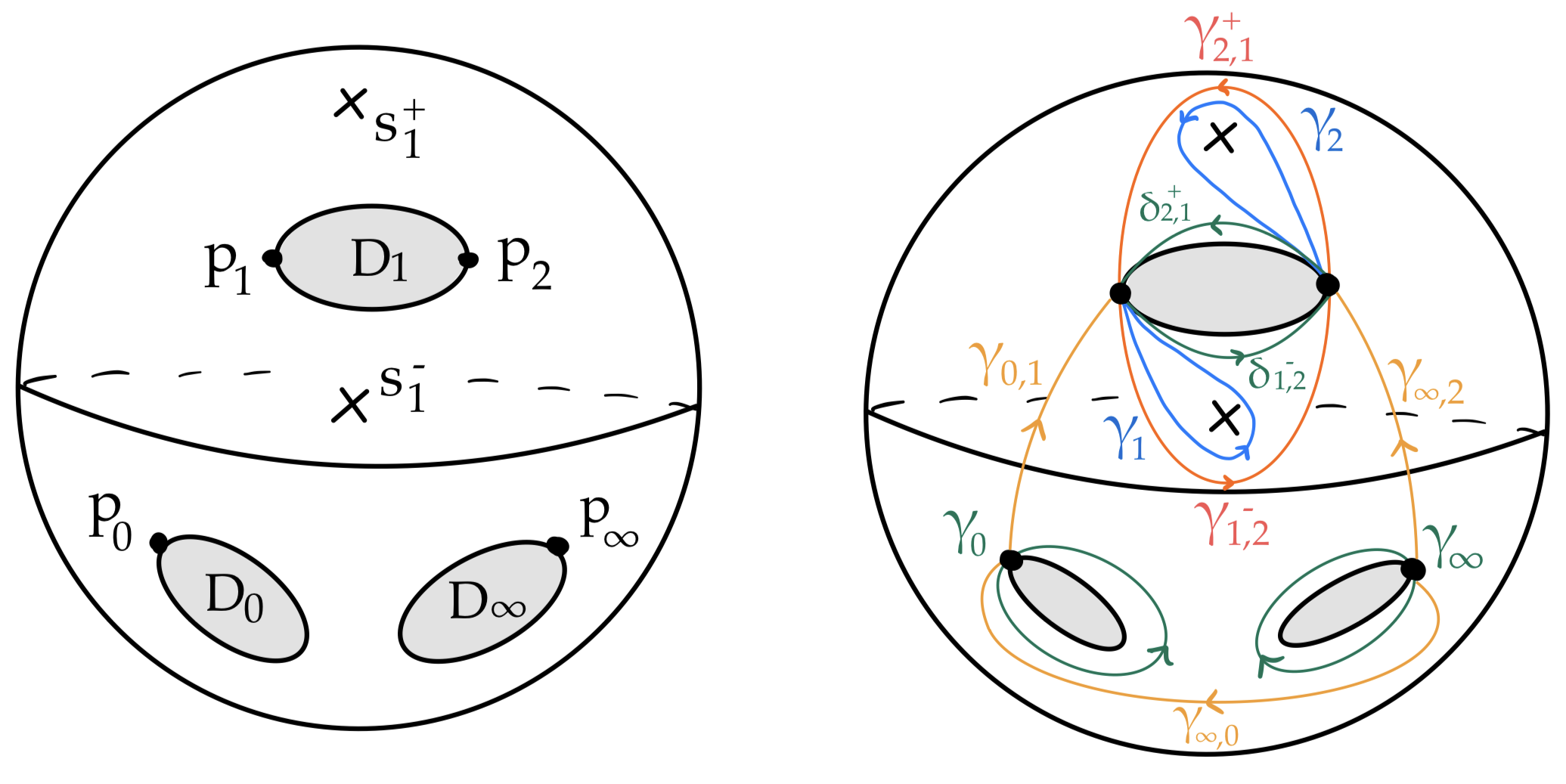}
	\end{center}
The relations between the generators of this presentation are generated by the following:
\begin{itemize}
	\item $(\gamma_{2,1}^+)^{-1}\star(\gamma_{\infty, 2})^{-1}\star\gamma_{\infty,0}\star\gamma_{0,1}=1$,
	\item $\gamma_{1,2}^-\star(\gamma_{\infty, 2})^{-1}\star\gamma_\infty\star\gamma_{\infty,0}\star\gamma_0\star\gamma_{0,1}=1$
	\item $\gamma_1\star\delta_{1,2}^-\star\gamma_2\star\delta_{2,1}^+=\gamma_{1,2}^-\star\gamma_{2,1}^+$
\end{itemize}
We have now enough data to study the monodromy representation of a PV connection:
\begin{align*}
	\rho_\nabla\colon\pi_1^V&\to \mathrm{GL}_2(\C)\\ [\gamma]&\mapsto M^\gamma.
\end{align*}
We recall that a basis change for $E$ produces an another monodromy representation conjugated to the initial one. In particular, up to basis changes, we can always achieve that
\[\rho_\nabla(\gamma_{\infty,2})=\rho_\nabla(\gamma_{\infty,0})=\rho_\nabla(\gamma_{0,1})=Id.\]
Note that this implies also that $\rho_\nabla(\gamma_{2,1}^+)=Id$.
\begin{defn}
	A representation of $\rho\colon\pi_1^V\to \mathrm{GL}_2(\C)$ is said \textit{normalized} if  $\rho(\gamma_{\infty,2})=\rho(\gamma_{\infty,0})=\rho(\gamma_{0,1})=\rho(\gamma_{2,1}^+)=Id$.
\end{defn}
We would like to keep the same structure that we found in the local fundamental groupoid.
\begin{defn}
	A normalized representation of $\rho\colon\pi_1^V\to \mathrm{GL}_2(\C)$ is said \textit{admissible} if 
	\begin{itemize}
		\item[$\bullet$] $\rho(\gamma_1):=S_1$ is upper triangular,
		\item[$\bullet$] $\rho(\delta_{1,2}^-\star\gamma_2\star(\delta_{1,2}^-)^{-1}):=S_2$ is lower triangular,
		\item[$\bullet$] $\rho(\delta_{1,2}^-\star\delta_{2,1}^+):=D$ is diagonal.
	\end{itemize}
\end{defn}
\begin{oss}
	We find back what we expected from the local fundamental groupoid. Indeed
	\begin{align*}
		\rho(\gamma_{1,2}^-\star\gamma_{2,1}^+)=&\;\rho\left(\gamma_1\star\delta_{1,2}^-\star\gamma_2\star\delta_{2,1}^+\right)=\\
		&\rho\left(\gamma_1\star\delta_{1,2}^-\star\gamma_2\star(\delta_{1,2}^-)^{-1}\star\delta_{1,2}^-\star\delta_{2,1}^+\right)=\\
		&\rho(\gamma_1)\cdot\rho(\delta_{1,2}^-\star\gamma_2\star(\delta_{1,2}^-)^{-1})\cdot\rho(\delta_{1,2}^-\star\delta_{2,1}^+)=\\
		&S_1\cdot S_2\cdot D.
	\end{align*}
\end{oss}
We denote by $\mathrm{Rep}^V_\Theta(\pi^V_1, \mathrm{GL}_2(\C))$ the set of normalized and admissible representations. Let us finally recall that we consider monodromy up to equivalence, that is up to global conjugation by a diagonal matrix, in order to preserve the admissibility.
\begin{defn}
	The \textit{wild character variety} associated to the fundamental groupoid of a rational irregular curve $(\P^1, D^V,J^V)$ is the categorical quotient
	\[\chi^V_\Theta:=\mathrm{Rep}^V_\Theta\left(\pi^V_1, \mathrm{GL}_2(\C)\right)//\mathcal D,\]
	where $\mathcal D$ is the subgroup of $\mathrm{GL}_2(\C)$ composed by diagonal matrices.
\end{defn}
In Section 3.2 of \cite{Saitovanderput}, the interested reader can find a geometric description of $\chi^V_\Theta$.

\subsection{Riemann-Hilbert Correspondence}
The Riemann-Hilbert correspondence for regular singular connections has been proved by Deligne \cite{Del} in 1970. The irregular case has been deeply studied by Philip Boalch \cite{Boalch2001GbundlesIA} in the framework of meromorphic connections. We can state it for PV connections in the following way:
\begin{thm}[Irregular RH for PV connections]\label{thm:irregRH}
    Let us consider the divisor $D^V:=[0]+2[1]+[\infty]$ and the set of jets $J^V_0=\{0, (1,t_0), \infty\}$ for a fixed complex number $t_0\in \C^*$. Then, the moduli space of PV connections on the rational irregular curve $\mathcal P_{t_0}:=(\P^1, D^V, J^V_0)$ with fixed generic spectral data $\Theta$, is analytically isomorphic to $\chi^V_\theta$ for any choice of $t_0\in \C^*$.
\end{thm}
\begin{oss}
    The moduli space of connections on $\mathcal P_{t_0}$ is just the hypersurface $Ok_{t_0}=\Conn\cap\{t=t_0\}$. In particular any connection in this space have a fixed $t=t_0$, but $q$ and $p$ that vary.
\end{oss}
\begin{oss}
    The fact that we consider spectral data $\exp(2i\pi\Theta)$ is the reason why we are interested to consider $\Theta$ just up to integer shifts.
\end{oss}
In other words, the irregular Riemann-Hilbert correspondence establish an isomorphism between the Okamoto spaces and the moduli space of representations of the wild fundamental group of $(\P^1, D^V, J)$. In formulae, for all $t_0\in \C^*$:
\[Ok_{t_0}\cong\chi^V_\Theta.\]

\section{Painlevé V Foliation}\label{sec:PVfol}
\subsection{Painlevé V Equation and its Hamiltonian System}\label{secpaineq}
One can wonder, given a generic class of PV connections, if we can deform it (that is, deform the parameters $(t,q,\hat p)$) without changing the monodromy. The irregular Riemann-Hilbert correspondence and the Painlevé property suggest that we will never have a positive answer until we keep $t$ constant. We already mentioned that the moduli space $\overline \Conn$ comes with a natural fibration $\pi\colon\overline\Conn\xrightarrow{t} \P^1$, decomposing it in an union of Okamoto spaces, that are the spaces of initial conditions for the Painlevé V equation, and the hypersurfaces corresponding to $t=0,\infty$. Theorem $\ref{thm:irregRH}$ states that each Okamoto space is analytically isomorphic to $\chi^V_\Theta$. We would then like to patch these isomorphisms together to construct a "global" Riemann-Hilbert map $RH\colon \overline\Conn\to\chi^V_\Theta$, but it turns out not to be possible for several reasons: first of all, we cannot define it on the boundary components $\pi^{-1}(0)$ and $\pi^{-1}(\infty)$. Moreover, even on $\Conn:=\pi^{-1}(\C^*)$, the Riemann-Hilbert map is not globally well defined. A monodromy action, coming from the topology of $\C^*$, arise, and therefore the Riemann-Hilbert map is well defined only as $RH\colon \pi^{-1}(U)\to \chi^V_\Theta$ for any simply connected open set $U\subseteq\C^*$.
	\begin{center}
		\includegraphics[width=0.45\linewidth]{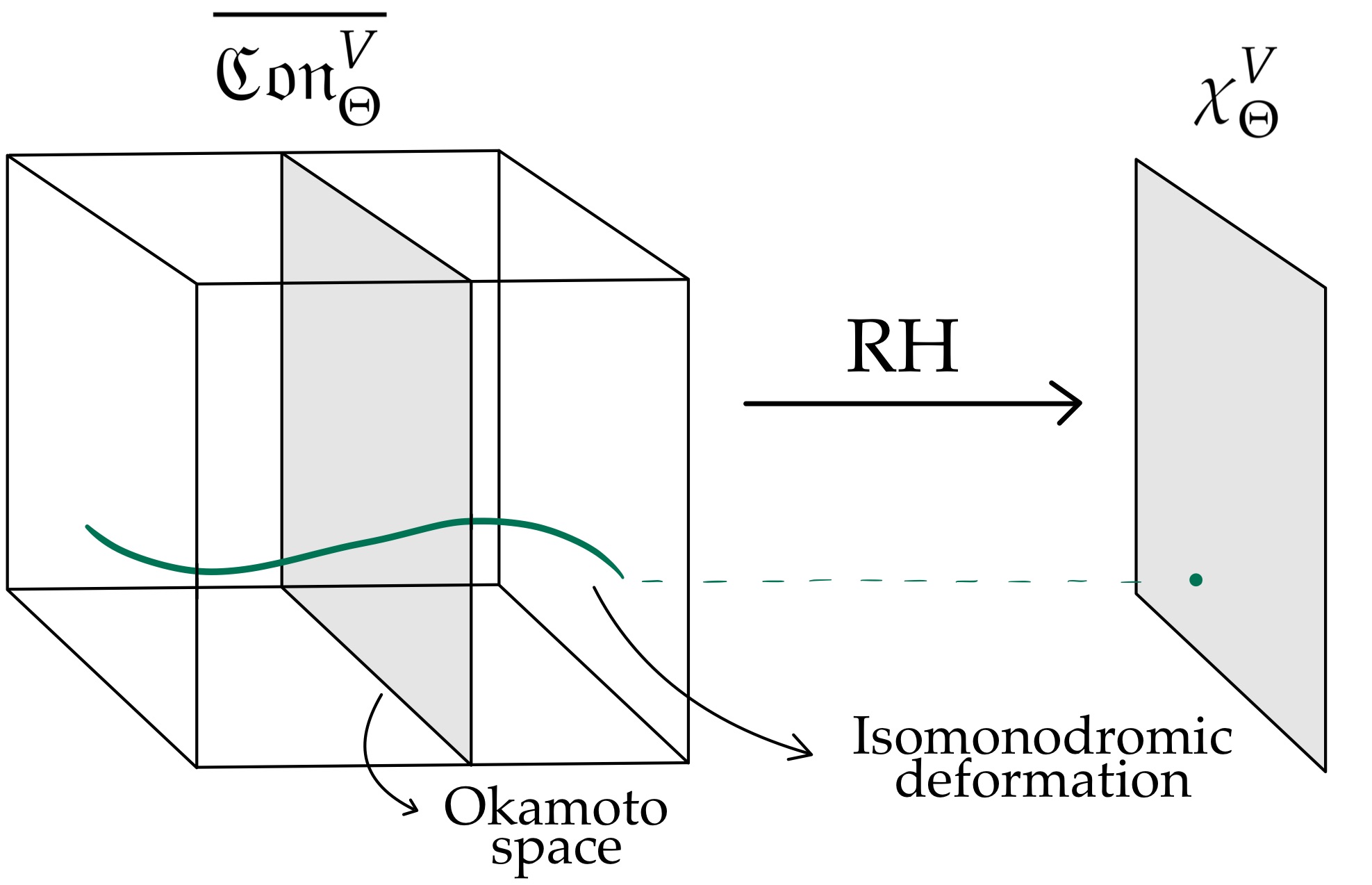}
	\end{center}

In 1981 Jimbo, Miwa and Ueno \cite{JIMBO1981306} shown that the fibers of the $RH$ map are parametrised, in the dense open set $\M\times\C$, as solutions of a very famous differential equation: the Painlevé V equation
\begin{equation}\label{eqV}
	\ddot q(t)=\Bigg(\frac{1}{2q(t)}+\frac{1}{q(t)-1}\Bigg)\dot q(t)^2-\frac1t\dot q(t)+\frac{(q(t)-1)^2}{t^2}\Bigg(\alpha q(t)+\frac{\beta}{q(t)}\Bigg)+\gamma \frac {q(t)}{t}+\delta\frac{q(t)(q(t)+1)}{q(t)-1}
\end{equation}
where
\begin{itemize}
	\item[$\bullet$] $(t,q)$ are the coordinates of $\M\subseteq\overline \Conn$,
	\item[$\bullet$] the fixed parameters $\{\alpha,\beta,\gamma,\delta\}$ can be expressed in terms of $\Theta=\{\kappa_0, \kappa_1,\kappa_\infty\}$.
\end{itemize}
Associated to the fifth Painlevé equation, there is an Hamiltonian function
\[H^V:=\frac{q \left(q -1\right)^{2} p^{2}-\left(\kappa_0 \left(q -1\right)^{2}+(\kappa_1-1)  q \left(q -1\right)-t q \right) p +\rho^V \left(q -1\right)}{t}\]
as described in \cite{Ohyama_2006}, Section 4.1. \\
If we fix a basis $\{\partial_t, \partial_q, \partial_p\}$ of the tangent bundle $T(\M\times\C)$, we can deduce the hamiltonian vector field 
\[X^{V}:=\partial_t+\displaystyle\frac{\partial H^V}{\partial p}\displaystyle\partial_q-\displaystyle\frac{\partial H^V}{\partial q}\displaystyle\partial_p\]
whose integral curves correspond to the solutions of the fifth Painlevé equation. We recall that a complex curve $\gamma\colon\C\to\M\times \C$, given by the expression $\gamma(t)=(t, q(t), p(t))$, is integral for $X^V$ if it satisfies 
	\[\dot\gamma(t)=X^V(\gamma(t)) \;\;\;\text{ that is }\;\;\;\begin{pmatrix}1\\ \dot q(t)\\ \dot p(t)\end{pmatrix}=\begin{pmatrix}1\\ \partial_pH^V(\gamma(t))\\-\partial_qH^V(\gamma(t))\end{pmatrix}.\]
\begin{oss}
	We can find back the fifth Painlevé equation via the following procedure: since $\partial_pH^V$ is linear in $p$, via the identity $\dot q=\partial_pH^V$ we can deduce a formula for $p$. Substituting this value in the identity $\dot p=-\partial_qH^V$ we get the fifth Painlevé equation as desired.
\end{oss}
\begin{oss}
	The relation of the parameter $\hat p$ appearing in \cite{Diarra, Matt}, and the parameter $p$ appearing in \cite{Ohyama_2006} for the hamiltonian function is the following:
	\[\hat p = -pq(q-1)^2.\]
	Moreover, if we call $\phi_p\colon \hat p\mapsto p$ this coordinate change, we get a strong relation between the hamiltonian $H^V$ and the parameter $\hat K$ appearing in the normal form:
		\[\hat K=t\cdot (H^V\circ\phi_p)+\rho^V+\frac{\hat p}{q-1}.\]
\end{oss}
We stress that the fifth Painlevé equation is defined by this formula only in the Zariski open set $\M\times \C\subseteq\overline \Conn$ corresponding to the moduli space of such PV connections with $t\neq0, \infty$ and $q\neq0,1,\infty$. We remark that in $\M\times \C$ the foliation induced by the hamiltonian vector field is regular. The goal of this work is then to understand the behaviour of the foliation in the boundary components of $\overline\Conn$.

\subsection{Vector Field and First Integrals}\label{secvectfield}
We will refer to boundary components as in Figure \ref{figmodsp}. Our goal is to extend the hamiltonian vector field in coordinates adapted to these boundary components and to identify first integrals that help describe the behavior of the foliation. The main idea can be resumed as follows: consider a leaf of the isomonodromic foliation and follow it as it approaches a boundary component. In this limit, the irregular curve degenerates to a nodal curve with two smooth components, and the corresponding PV connection splits accordingly into two connections, one on each component. Both connections acquire a pole in the node of the curve. The local monodromy at the node should, in some sense, encode the monodromy of the original connection and therefore remain invariant along the isomonodromic leaf. In other words, we expect that the residual spectral datum in the node (that we computed in \cite{Matt}) provides a first integral for the hamiltonian vector field. We recall that we will use the following coordinates:
\begin{center}
    \includegraphics[width=12cm]{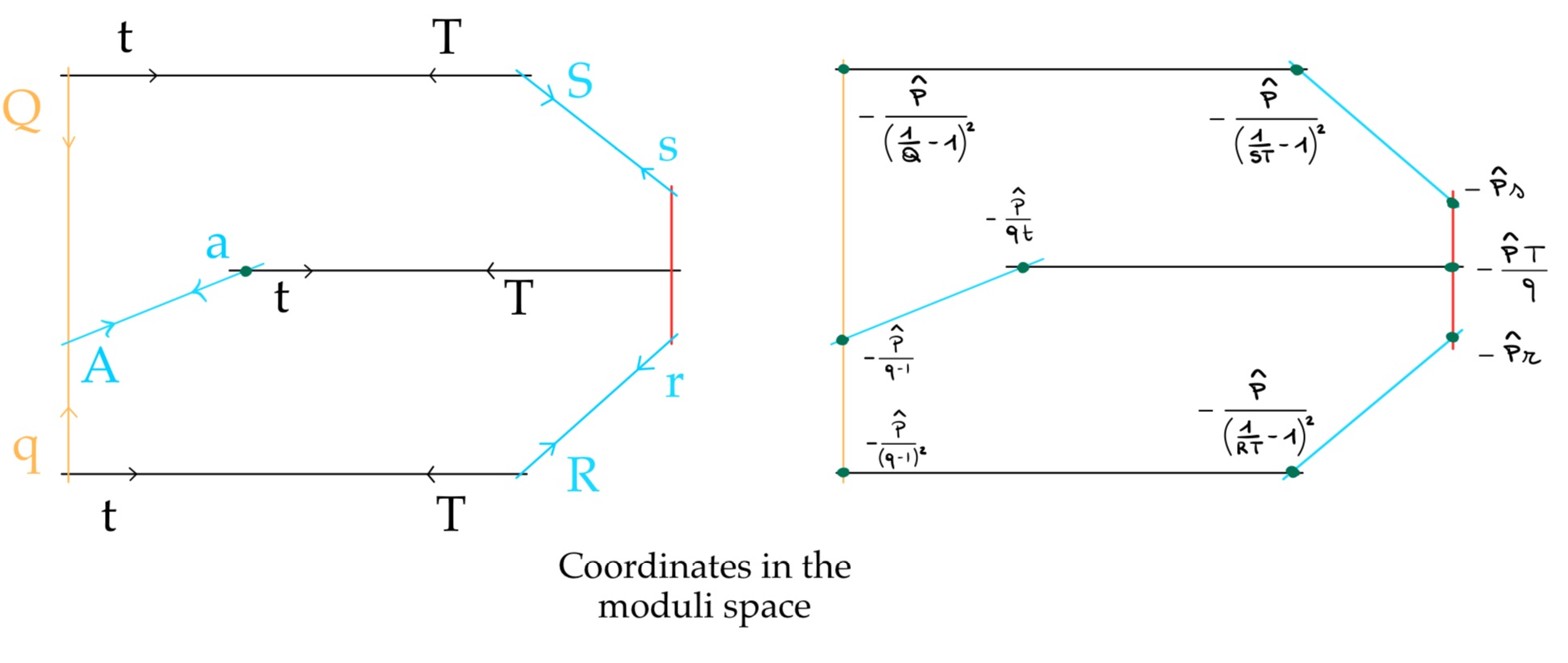}
\end{center}
where,
\[T=\frac1t,\;\;\;\;\; Q=\frac1q,\;\;\;\;\;A=\frac{t}{q-1}=\frac1a,\;\;\;\;\; S=\frac{T}{Q}=\frac1s,\;\;\;\;\; R=\frac qT=\frac1r\]
For simplicity we call:
\begin{align*}
    &P^0=-\frac{\hat p}{(q-1)^2}=-\frac{\hat p}{(RT-1)^2},\;\;\;\;\; P^1=-\frac{\hat p}{qt}=-\frac{\hat pT}{q}=-\hat p s=-\hat pr, \\&P^\infty=-\frac{\hat p}{(\frac{1}{Q}-1)^2}=-\frac{\hat p}{(\frac{1}{ST}-1)^2},\;\;\;\;\; P^1_0=-\frac{\hat p}{q-1}.
\end{align*}
We proceed case by case. 
We first study $\mathcal Q^0, \mathcal Q^1$ and $\mathcal Q^\infty$. Recall that the line bundle $\mathcal O_{\widetilde{\mathcal M}}(D)$ is trivial when restricted to each of these components. From Okamoto \cite{okamoto} and from computations in \cite{Matt}, we deduce that no connections lie in the $\mathcal Q^i$s but rather on the special sections we blow up. Nevertheless, we wish to understand the behaviour of the isomonodromic foliation along these components. We will resume in Theorem \ref{thm:Qi} the main results.
\medskip
\textbf{Vector Field on $\mathcal{Q}^0$.} First of all we need to extend the hamiltonian vector field to $\mathcal{Q}^0$. Let us define $f_{\mathcal Q_0}$ the coordinate change $(t,q,p)\mapsto(t,q,P^0)$. We define
\[X^V_{\mathcal Q^0}(t, P^0):=\lim_{q\to 0}(f_{\mathcal Q_0*}X^V)(t, q, P^0)=P^0(P^0-\kappa_0)\partial_{P^0},\]
the limit vector field restricted to the hypersurface $\mathcal Q^0$. It is a vertical vector field that is singular along the two special sections $P^0=0$ and $P^0=\kappa_0$, as we expected.
We recall that on such exceptional divisor, that we call $\mathcal E^0_\pm$, we set respectively the coordinates $\alpha^0_\pm$ given by the expression 
\[\alpha^0_-=\frac{q}{P^0}=-\frac{q(q-1)^2}{\hat p}=p\;\;\;\;\;\;\text{ and }\;\;\;\;\alpha_+^0=\frac{q}{P^0-\kappa_0}.\]
Changing now coordinates of $X^V$ by $(t,q,p)\to(t,q,\alpha^0_\pm)$ we get the following vector fields 
\begin{align*}
	X^V_{\mathcal E^0_-}(t,q,\alpha^0_-)&=t\partial_t\\&
	+[2\alpha^0_- q^3-(4\alpha^0_- +\kappa_0+\kappa_1-1)q^2+(t+2\alpha^0_-+2\kappa_0+\kappa_1-1)q-\kappa_0]\partial_q\\&
	+[(-3q^2+4q-1)(\alpha^0_-)^2+(-2(\kappa_0-\kappa_1+1)q-t-2\kappa_0-\kappa_1+1)\alpha^0_-+\\&+(\kappa_1-1)\kappa_0-\rho]\partial_{\alpha^0_-}
\end{align*}
\begin{align*}
	X^V_{\mathcal E^0_+}(t,q,\alpha^0_+)&=t\partial_t\\&
	+[2\alpha^0_+ q^3+(-4\alpha^0_+ +\kappa_0-\kappa_1+1)q^2+(t+2\alpha^0_+-2\kappa_0+\kappa_1-1)q+\kappa_0]\partial_q\\&
	+[(-3q^2+4q-1)(\alpha^0_+)^2+(-2(\kappa_0-\kappa_1+1)q-t+2\kappa_0-\kappa_1+1)\alpha^0_++\\&+(\kappa_1-1)\kappa_0-\rho]\partial_{\alpha^0_+}
\end{align*}

As computed in \cite{Matt}, the residual spectral datum in the node is, respectively $(\kappa_0\pm1)^2$, that, being constant, gives no information about the isomonodromic foliation. What it is important to notice is that the exceptional divisor is not invariant, since $X^V_{\mathcal E^0_\pm}|_{\mathcal E^0_\pm}$ has a non zero coefficient for $\partial_q$, and also that the coefficient of $\partial_t$ is a unity for $t\in \C^*$, that is part of the Painlevé property.

\medskip
\textbf{Vector Field on $\mathcal{Q}^1$.} Let us define $f_{\mathcal Q^1}$ the coordinate change $(t,q,p)\mapsto(t,a,P^1)$. The hamiltonian vector field restricted to $\mathcal Q^1$ is then
\[X^V_{\mathcal Q^1}(t, P^1)=\lim_{a\to 0}(f_{\mathcal Q^1*}X^V)(t, q, P^1)=2P^1(P^1+1)\partial_{P^1},\]
that, as expected, is vertical and presents singularities along the special sections $P^1=0,-1$.
\\As for $\mathcal Q^0$, we consider the vector field in the coordinates of the exceptional divisors. We recall that, as we can see in Figure \ref{figmodsp}, over $\mathcal Q^1$ we have to blow-up twice the special sections. We call $\mathcal E^1_\pm$ the exceptional divisors that are part of the moduli space, that are arising from the second blow up. We call $\mathcal F^1_\pm$ the intermediate ones. We will call by $\beta^1_\pm$ the coordinates on $\mathcal F^1_\pm$ and by $\alpha^1_\pm$ the coordinates on $\mathcal E^1_\pm$. They are defined as follows:
\[P^1=\beta^1_-(q-1) \;\;\;\text{ and }\;\;\;P^1=\frac{\alpha^1_-}{t^2}(q-1)^2,\]
that means that we blow-up the line $\beta^1_-=0$ to obtain $\mathcal E^1_-$. And
\[P^1=\beta^1_+(q-1)-1 \;\;\;\text{ and }\;\;\;P^1=\frac{\alpha^1_+}{t^2}(q-1)^2+\frac{\kappa_1}{t}-1,\]
that means that we blow up the line $\beta^1_-=\kappa_1/t$ to obtain $\mathcal E^1_+$.\\
For the first exceptional divisor, we change the coordinates of $X^V$ by $(t,q,p)\to(t,q,\beta^1_\pm)$ getting the following vector fields
\[(X^V_{\mathcal F^1_-})_{|\mathcal F^1_-}(t,\beta^1_-)=\beta^1_-t^2\partial_{\beta^1_-}\]
\[(X^V_{\mathcal F^1_+})_{|\mathcal F^1_+}(t,\beta^1_+)=t(-\beta^1_+t+\kappa_1)\partial_{\beta^1_+}.\]
We remark that they are stationary and vanish on the lines $\beta^1_-=0$ and $\beta^1_+=\kappa_1/t$, that are exactly the lines we want to blow-up. On the second exceptional divisor we get then to the vector fields, restricted on them, have the following expression 
\begin{align*}
	X^V_{\mathcal E^1_-}(t,q,\alpha^1_-)&=t^2\partial_t\\&
	+[2\alpha^1_-q^3+((-\kappa_0-\kappa_1+1)t-4\alpha^1_-)q^2-(t^2-(2\kappa_0+\kappa_1-1)t-2\alpha^1_-)q-\kappa_0t]\partial_q\\&
	+[(-3q^2+4q-1)(\alpha^1_-)^2+(t+2q(\kappa_0+\kappa_1-1)-2\kappa_0-\kappa_1+2)t\alpha^1_--\\&-\rho t^2]\partial_{\alpha^1_-} 
\end{align*}
\begin{align*}
	X^V_{\mathcal E^1_+}(t,q,\alpha^1_+)&=t^2\partial_t\\&+[2\alpha^1_+q^3+((\kappa_1-\kappa_0+1)t-4\alpha^1_+)q^2-(t^2-(2\kappa_0-\kappa_1-1)t-2\alpha^1_+)q-\kappa_0t]\partial_q\\&+[(-3q^2+4q-1)(\alpha^1_+)^2+(t+2q(\kappa_0-\kappa_1-1)-2\kappa_0+\kappa_1+2)t\alpha^1_++\\&+((\kappa_0-1)\kappa_1-\rho)t^2]\partial_{\alpha^1_+} 
\end{align*}
As computed in \cite{Matt}, the residual spectral data in the node are, respectively, $(\kappa_1\pm1)^2$, that, being constant, gives no information about the isomonodromic foliation. What it is important to notice is that the exceptional divisor is not invariant, since $X^V_{\mathcal E^1_\pm}|_{ \mathcal E^1_\pm}$ has a non zero coefficient for $\partial_q$, and also that the coefficient of $\partial_t$ is a unity for $t\in \C^*$, that is part of the Painlevé property. 

\medskip
\textbf{Vector Field on $\mathcal{Q}_\infty$.} Thanks to the symmetries of the moduli space, and of the PV equation, we know that we get the same behaviour as on $\mathcal Q^0$.

\medskip

Before passing to the study of the isomonodromic foliation on the $\mathcal B^i_j$, let us resume what we found in the following result.

\begin{thm}\label{thm:Qi}
	The $\mathcal Q^i$s are invariant for the isomonodromic foliation, moreover the hamiltonian vector field is therein vertical, that is, parallel to $\partial_{P^i}$. For $q=1$, also the intermediate exceptional divisor $\mathcal F^1_\pm$ is invariant, and the foliation is vertical. On the other hand, the exceptional divisors $\mathcal E^i_\pm$ are, as expected \cite{Heu2019FlatRT}, transverse to the foliation.
\end{thm}
\begin{center}
		\includegraphics[width=0.3\linewidth]{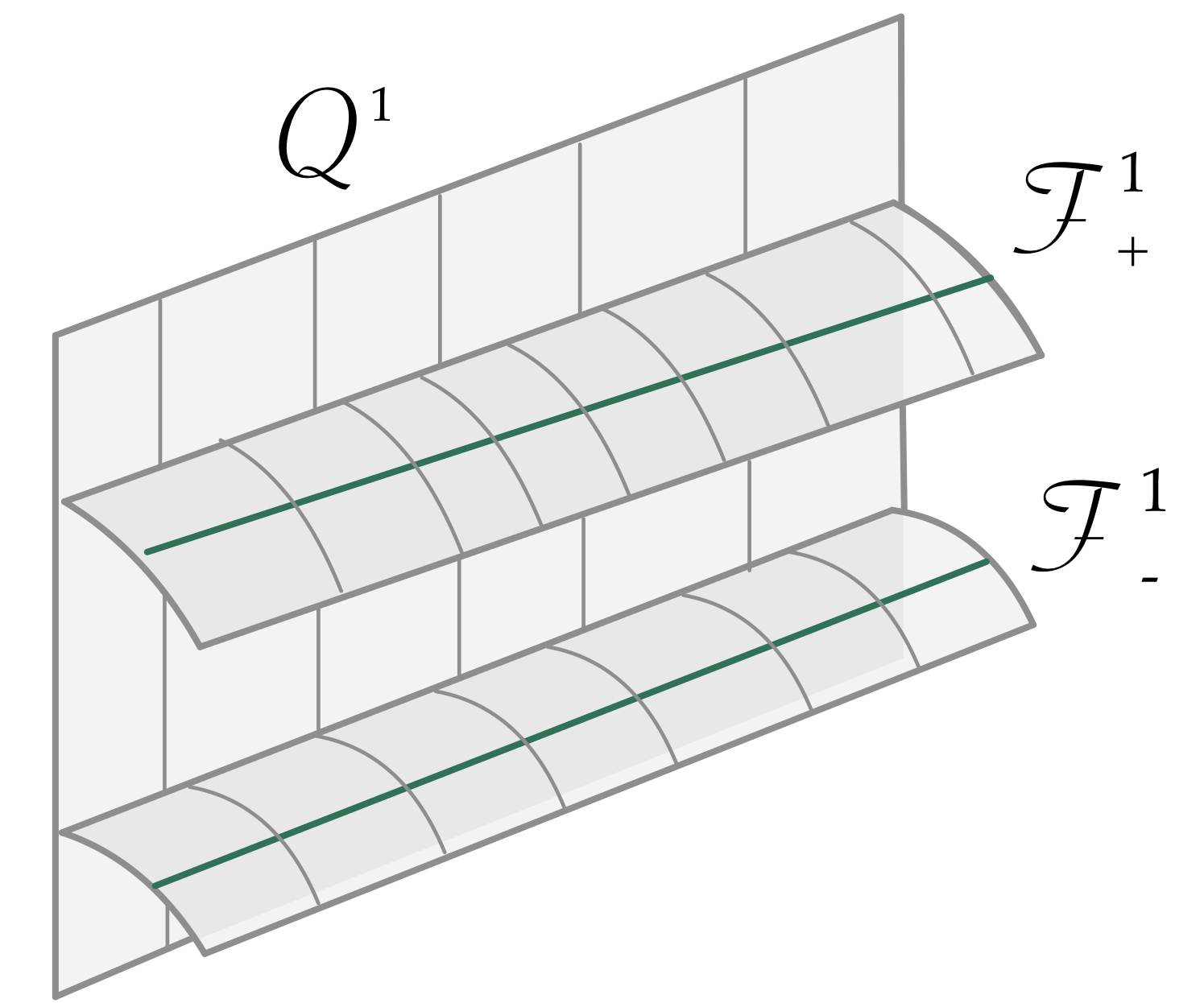}
        \captionof{figure}{}\label{fig:folQ1}
\end{center}

The hypersurfaces $\mathcal B^i_j$ turn out also to be invariant for the isomonodromic foliation. Since the line bundle is not trivial in restriction to them, we are not able to find a global vertical coordinate, as we did for the $\mathcal Q^i$s. We remark that for any $i$, the intersection $\mathcal B^i_j\cap \mathcal Q^i$ is a projective line. In particular we are interested in studying the behaviour of the foliation near the intersection of $\mathcal B^i_j$ and the special sections of $\mathcal Q^i$.

\medskip

\textbf{Vector Field and First integral on $\mathcal{B}_0^1$.} Let us define $f_{\mathcal B^1_0}$ the coordinate change $(t,q,p)\mapsto(A,\tilde q,P^1_0)$, where $\tilde q =q-1$, centered in $A_0\cap B^1_0$, and $g_{\mathcal B^1_0}$ the coordinate change $(t,q,p)\mapsto(t,a,P^1)$, centered in $Q^1\cap B^1_0$. The hamiltonian vector field restricted to $\mathcal B_0^1$ is then
\[X^V_{\mathcal B_0^1}(A, P^1_0)=\lim_{\tilde q\to 0}(f_{\mathcal B_0^1*}X^V)(A, \tilde q, P^1_0)=-AP^1_0\partial_{P^1_0}+(A+2P^1_0-\kappa_1)\partial_{A},\]
or
\[X^V_{\mathcal B_0^1}(a, P^1)=\lim_{t\to 0}(g_{\mathcal B_0^1*}X^V)(t, a, P^1)=a(-\kappa_1a+2P^1+1)\partial_{a}-P^1(\kappa_1a+2P^1+2)\partial_{P^1}.\]
We remark that the vector field becomes vertical at the intersection with $\mathcal Q^1=\{a=0\}$.
\\The residual spectral data in the node are 
\[(\kappa_1+1)^2+4P^1_0(A+P^1_0-\kappa_1)\;\;\;\;\;\text{ and }\;\;\;\;\;(\kappa_1+1)^2-\frac{4P^1\kappa_1}{a}+\frac{4P^1(P^1+1)}{a^2}.\]
\begin{prop}
	The residual spectral data are first integrals for the respective vector fields.
\end{prop}
\begin{proof}
	A computation shows
	\[X^V_{\mathcal B_0^1}(A, P^1_0)\Big((\kappa_1+1)^2+4P^1_0(A+P^1_0-\kappa_1)\Big)=0,\]and\[ X^V_{\mathcal B_0^1}(a, P^1)\Big((\kappa_1+1)^2-\frac{4P^1\kappa_1}{a}+\frac{4P^1(P^1+1)}{a^2}\Big)=0.\]
	By the definition of first integral, we conclude.
\end{proof}
\begin{prop}
	The isomonodromic leaves in $\mathcal B^1_0$ are contained in the level sets of the first integrals and they form a pencil of conics.
\end{prop}
\begin{proof}
	The first part of the theorem is trivial. In order to state that the curves are of genus zero we use the Riemann-Hurwitz formula. We consider the level sets 
	\[\mathcal C_b:=\{(\kappa_1+1)^2+4P^1_0(A+P^1_0-\kappa_1)=b\}\]
	for $b\in \C$. They are curves in $\mathcal B^1_0\cong\F_1$. The projection $\pi\colon\F_1\to \P^1$ restricted to $\mathcal C_b$ is a ramified covering. The branching point are the roots of the discriminant
	\[\frac{\Delta_b}{16}=(A-\kappa_1)^2+b-(\kappa_1+1)^2=A^2-2\kappa_1A+b-2\kappa_1+1.\]
	For a generic value of $b$ there are two roots defining the two branching points of the covering. We then know that $R=\infty$ is not a branching point since they must be even in number. Applying the Riemann-Hurwitz formula we get then 
	\[2g_{\mathcal C_b}-2=-4+\sum_{p\in \mathcal C_a}(e_p-1)=-4+2=-2\]
	deducing $g_{\mathcal C_b}=0$ as desired. 
\end{proof}
A study of the equation for the first integral around $B^1_0\cap Q^1$ indeed shows that each conic is tangent at the two points $(t,a,P^1)=(0,0,0)$ and $(t,a,P^1)=(0,0,-1)$, that coincide with the intersections between the two special sections in $\mathcal Q^1$ and the hypersurface $\mathcal B^1_0$. Their tangent slopes at these points are 0 and $\kappa_1$, respectively. These values correspond to the coordinates of the sections we blow up in the intermediate divisors $\mathcal F^1_\pm$ in $\mathcal Q^1$.

The foliation is shown in Figure \ref{fig:folA0}

\medskip

\textbf{Vector Field and First integral on $\mathcal{B}_\infty^0$.}
Let us define $f_{\mathcal{B}_\infty^0}$ the coordinate change $(t,q,p)\mapsto(T, R, P^0)$ and $_{\mathcal{B}_\infty^0}$ the coordinate change $(t,q,p)\mapsto(r, q, P^1)$. The hamiltonian vector field restricted to $\mathcal{B}_\infty^0$ is then
\[X^V_{\mathcal{B}_\infty^0}(R, P^0)=\lim_{T\to 0}(f_{\mathcal{B}_\infty^0*}X^V)(T, R, P^0)=-R(2P^0+R-\kappa_0)\partial_R-P^0(P^0-\kappa_0)\partial_{P^0},\]
or
\[X^V_{\mathcal{B}_\infty^0}(r, P^1)=\lim_{q\to 0}(g_{\mathcal{B}_\infty^0*}X^V)(r, q, P^1)=-r(2P^1-\kappa_0r+1)\partial_r-P^1(P^1+1)\partial_{P^1}.\]
We remark that the vector field becomes vertical at the intersection with $\mathcal Q_0=\{R=0\}$ and $\mathcal A_\infty=\{r=0\}$ respectively.\\
In \cite{Matt} we show that the residual spectral data in the node are not the same for the two curves, but their difference is $\kappa_1$, that is a constant. They are the following:
\[2P^0-\kappa_0-1+\frac{2P^0(P^0-\kappa_0)}{R}\;\;\;\;\;\text{ and }\;\;\;\;\;-2P^1\kappa_0-\kappa_0-1+\frac{2P^1(P^1+1)}{r}.\]
We can therefore state the following proposition.
\begin{prop}\label{prop:Hur}
	The residual spectral data are first integrals for the respective vector fields.
\end{prop}
\begin{proof}
	We consider just one of the residual data. A computation shows
	\[X^V_{\mathcal{B}_\infty^0}(R, P^0)\Bigg(2P^0-\kappa_0-1+\frac{2P^0(P^0-\kappa_0)}{R}\Bigg)=0,\] and \[ X^V_{\mathcal{B}_\infty^0}(r, P^1)\Bigg(-2P^1\kappa_0-\kappa_0-1+\frac{2P^1(P^1+1)}{r}\Bigg)=0.\]
	By the definition of first integral, we conclude.
\end{proof}
\begin{prop}
	The isomonodromic leaves in $\mathcal B^0_\infty$ are contained in the level sets of the first integrals and they form a pencil of conics.
\end{prop}
\begin{proof}
	The first part of the theorem is trivial. 
	In order to state that the curves are of genus zero we use the Riemann-Hurwitz formula. We consider the level sets 
	\[\mathcal C_a:=\{R(2P^0-\kappa_0-1-a)+2P^0(P^0-\kappa_0)=0\}\]
	for $a\in \C$. They are curves in $\mathcal B^0_\infty\cong\F_1$. The projection $\pi\colon\F_1\to \P^1$ restricted to $\mathcal C_a$ is a ramified covering. The branching point are the roots of the discriminant
	\[\frac{\Delta_a}{4}=R^2-R(3\kappa_0+2+2a)+\kappa_0^2.\]
	For a generic value of $a$ there are two roots defining the two branching points of the covering. We then know that $R=\infty$ is not a branching point since they must be even in number. Applying the Riemann-Hurwitz formula we get then 
	\[2g_{\mathcal C_a}-2=-4+\sum_{p\in \mathcal C_a}(e_p-1)=-4+2=-2\]
	deducing $g_{\mathcal C_a}=0$ as desired.
\end{proof}
A study of the equation for the first integral around $B^0_\infty\cap Q^0$ shows that these conics pass through the two points $(T,R,P^0)=(0,0,0)$ and $(T,R,P^0)=(0,0,\kappa_0)$, that coincide with the intersection points between the two special sections in $\mathcal Q^0$ and the hypersurface $\mathcal B^0_\infty$. 

We now study the behaviour of the foliation in $\mathcal A_0$ and $\mathcal A_\infty$. 

\medskip
\textbf{Vector Field and First integral on $\mathcal{B}_\infty^\infty$.}
Thanks to the symmetries of the moduli space, and of the PV equation, we know that we get the same behaviour than on $\mathcal B_\infty^0$.

\medskip
\textbf{Vector Field and First integral on $\mathcal{A}_0$.} Thanks to the symmetry exchanging $\mathcal{Q}^0$ and $\mathcal Q^\infty$, we will only study the foliation in a neighborhood of $\mathcal A_0\cap\mathcal Q^0$ and $\mathcal A_0\cap\mathcal B^1_0$.\\On the first intersection we can apply the same coordinate change we used on $\mathcal Q_0$, that is $f_{\mathcal Q_0}\colon (t,q,p)\mapsto(t,q,P^0)$.  
The hamiltonian vector field restricted to $\mathcal A_0$ is then 
\begin{align*}
	X^V_{\mathcal A_0}(q,P^0)&=\lim_{t\to0}(f_{\mathcal Q^0*}X^V)(t,q,P^0)=\\&= 2q(q-1)\Big(\left(P^0-\frac{\kappa_0}{2}\right)(q-1)-\frac{\kappa_1-1}{2}q\Big)\partial_q+\\&+\Big(\left((P^0)^{2}-P^0( \kappa_0+\kappa_1-1) +\rho \right) q^{2}-P^0(-P^0+ \kappa_0)\Big)\partial_{P^0}.
\end{align*}
We remark that it becomes vertical, as expected, for $q=0$ and moreover it vanishes for $P^0=0,\kappa_0$.\\
On the other intersection $\mathcal A_0\cap\mathcal B^1_0$, we must define a new coordinate change $f_{\mathcal A_0}\colon (t,q,p)\mapsto(A,\tilde q, P^1_0)$. The vector field then becomes
\begin{align*}
	X^V_{\mathcal A_0}(\tilde q,P^1_0)&=\lim_{A\to0}(f_{\mathcal A_0*}X^V)(A,\tilde q,P^1_0)=\\&=
	\tilde q \left(\tilde q +1\right) \left(\mathit{\kappa_0} \tilde q +(\kappa_1-1)(\tilde q +1)-2 P^1_0 \right) \partial_{\tilde q} +\tilde q\left( (\tilde q+1)^2\rho-P^1_0(P^1_0+\kappa_0)\right) \partial_{P^1_0}.
\end{align*}
The residual spectral data in the node, in the new coordinates, are respectively

\[q\left((P^0)^2-(\kappa_0+\kappa_1-1)P^0+\rho\right)-2(P^0)^2+(2\kappa_0+\kappa_1-1)P^0-\rho+\left(\frac{\kappa_1}{2}\right)^2+\frac{P^0(P^0-\kappa_0)}{q}\]

\[\frac{4 P^1_0 \left(P^1_0 +\kappa_0\right)}{\tilde q +1}+4 ( -\kappa_0- \kappa_1+1)P^1_0 +4 \rho \tilde q +\kappa_1^2.\]
\begin{prop}
	The residual spectral data are first integral for the respective vector fields.
\end{prop}
\begin{proof}
	By the definition of first integral, we conclude via a direct computation.
\end{proof}
\begin{prop}
	The isomonodromic leaves in $\mathcal A_0$ are contained in the level sets of the first integrals and they form a pencil of conics.
\end{prop}
\begin{proof}
	The first part of the theorem is trivial. The Riemann-Hurwitz formula gives us the second part (as in Proposition \ref{prop:Hur}). It is indeed easy to see that all the conics pass through the points $(t,q,P^0)=(0,0,0)$ and $(t,q,P^0)=(0,0, \kappa_0)$, that coincide with the intersection points between the two special sections in $\mathcal Q^0$ and the hypersurface $\mathcal A_0$. 
\end{proof}

We represented in Figure \ref{fig:folA0} the isomonodromic foliation. Grey lines represent respectively the intersections $\mathcal A_0\cap \mathcal Q^0$ and $\mathcal B^1_0 \cap \mathcal Q^1$.
\begin{center}
    \includegraphics[width=8cm]{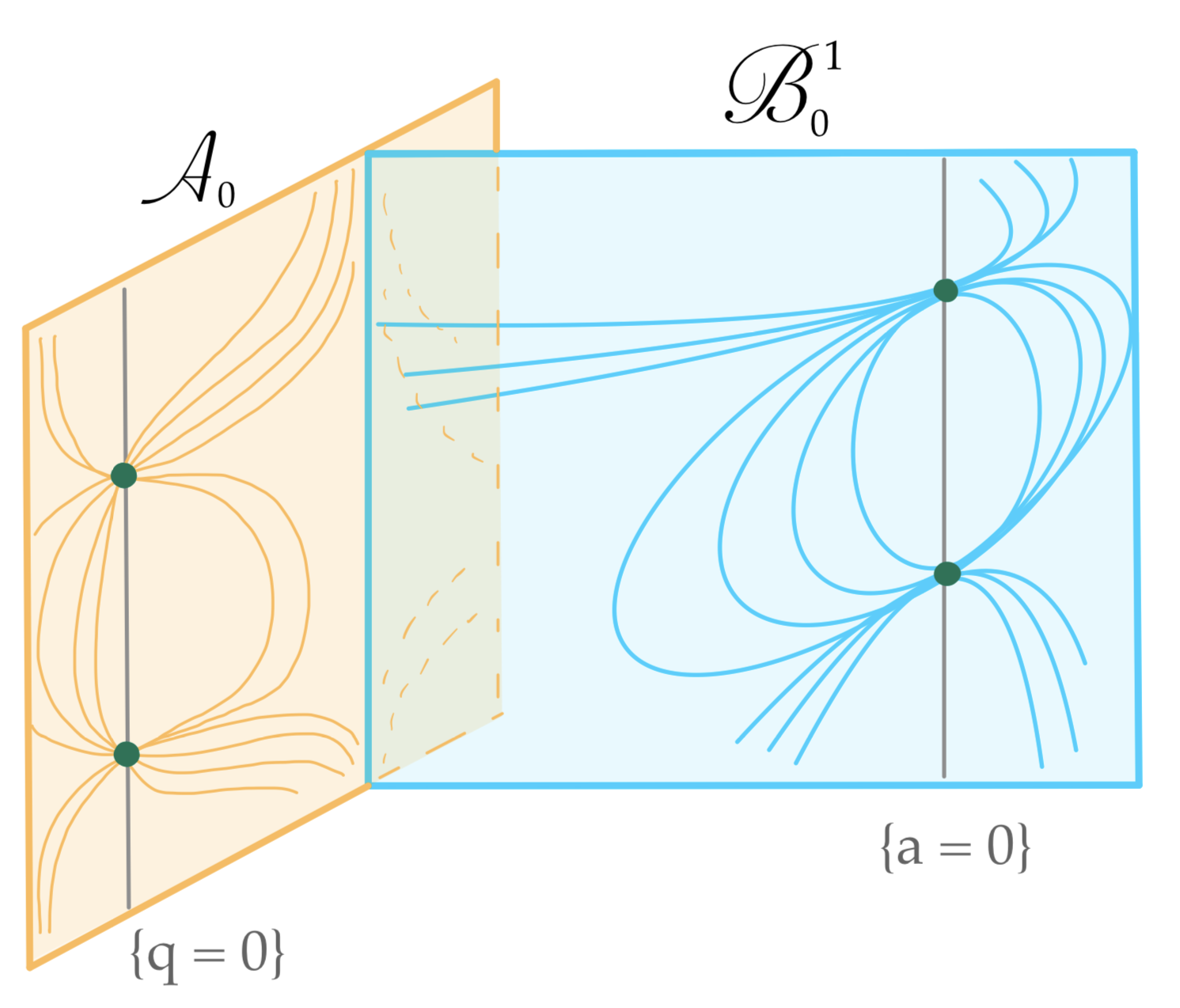}
    \captionof{figure}{}\label{fig:folA0}
\end{center}

\medskip

\textbf{Vector Field and First integral on $\mathcal{A}_\infty$.} As shown in \cite{Matt}, no PV connection lies on this boundary component, exactly as on the $\mathcal Q^i$s. Let us indeed recall that on $\mathcal A_\infty$ we had a choice to do regarding the local coordinates. 
\begin{itemize}
	\item[$\bullet$] The first possible coordinate is $P^1$, the same that we use on $\mathcal Q^1$. It is the one that trivializes the line bundle $\O_{\widetilde \M}(D)$ over $\mathcal A_\infty$. In \cite{Matt} it is shown that, using $P^1$ as a parameter for the PV connections, the limit $t\to\infty$ does not lead to some well defined connection. Moreover, the limit irregular nodal curve that we get under the associated Moebius transformation is not stable.
	\item[$\bullet$] Then there are coordinates $P_\infty^\pm$. Geometrically they correspond to the coordinates on the exceptional divisor arising by the blow up of the sections $P^1=0$ and $P^1=-1$ in $\mathcal A_\infty$. With these parameters, as computed in \cite{Matt}, the limit $t\to\infty$ produces well defined connections and a well defined stable nodal curve. 
\end{itemize}
Since the core of this work are isomonodromic deformations of PV connection, let us firstly analyse the most natural case, studying the Hamiltonian vector field in the coordinates $P_\infty^\pm$. Then we will see that also the coordinate $P^1$ give rise to a very interesting foliation on $\mathcal A_\infty$ and we will give to the objects lying in it the interpretation as Higgs bundles.

Let us recall that 
\[P_\infty^+=-\frac{\hat p}{(q-1)^2}=\frac{ P^1q}{T(q-1)^2}\;\;\;\text{ and }\;\;\;P_\infty^-=-\frac{\hat p}{(q-1)^2}+\frac{T\kappa_1(q-1)-q}{T(q-1)^2}=\frac{q(P^1+1)}{T(q-1)^2}-\frac{\kappa_1}{q-1}\]
are the coordinates on the exceptional divisors relative to the blow-up of the two sections $P^1=0$ and $P^1=-1$.

To compute the vector field restricted to these exceptional divisors we will use the set of coordinates $(0, q, P_\infty^\pm)$, getting to
\[X^V_{\mathcal A_\infty^+}(q, P_\infty^+)=-q^2\partial_q\]
\[X^V_{\mathcal A_\infty^-}(q, P_\infty^-)=q^2\partial_q\]
that is hence horizontal, as shown in orange in Figure \ref{fig:foliation}. This result is not really surprising, since the connections lying on an integral curve are exactly the same (they are not only sharing the same monodromy) as shown in Lemmas 4.19 and 4.20 of \cite{Matt}.

\medskip

In the coordinate $P^1$ we can still consider a particular well defined limit that give rise to something different from a PV connection. 
\begin{prop}\label{higgs}
    Let $(\O\oplus\O(2),D^V+[q],\nabla)$ be a PV connection in its normal form as in \cite{Diarra} and \cite{Matt}. Let $\Omega_0$ its connection matrix expressed in the chart $U_0:=\P^1\setminus\{\infty\}$ and let us consider the constant matrix
    \[M=\begin{pmatrix}1&0\\0&t\end{pmatrix}\in \mathrm{GL}_n(\C),\]
    and call $\mathcal G_M$ the associated (constant) gauge transformation. The limit 
    \[\lim_{t\to \infty}\frac1t\left(\mathcal G_M\cdot\nabla\right)\]
    is well defined and endows $\O\oplus\O(2)$ 
    with a Higgs bundle structure $(\O\oplus\O(2), \Omega_{A_\infty})$ with $\Omega_{A_\infty}\in \mathfrak{gl}(\Omega^1(D^V))$ defined, in the local chart $U_0$, by the expression 
    \[
    \Omega_{A_\infty|U_0}:=\lim_{t=\infty}\Bigg[\frac1t\cdot\left(M^{-1}\cdot\Omega_0(t,q,P^1;x)\cdot M\right)\Bigg]=\begin{pmatrix}0 & \frac{1}{\left(x -1\right)^{2} x} 
\\
 \frac{\left(P^1 +1\right) P^1 q}{\left(q -1\right)^{2}} & \frac{1}{\left(x -1\right)^{2}} 
\end{pmatrix}
\]
\end{prop}
\begin{proof}
    In the chart $U_0$ the PV connection $\nabla$ can be expressed as $d+\Omega_0$ and the limit then becomes 
    \[\lim_{t\to \infty}\frac1t\left(\mathcal G_M\cdot\nabla\right)\;\stackrel{\text{on } U_0}{=}\;\lim_{t\to \infty}\frac1t\left(d+ M^{-1}\Omega_0 M\right)=\cancel{\lim_{t\to \infty}\frac1t d}+\lim_{t\to \infty}\frac1t\left(M^{-1}\Omega_0 M\right)=\begin{pmatrix}0 & \frac{1}{\left(x -1\right)^{2} x}\\
 \frac{\left(P^1 +1\right) P^1 q}{\left(q -1\right)^{2}} & \frac{1}{\left(x -1\right)^{2}} 
\end{pmatrix}\]
    where the last equality comes from direct computations. We get then that $(\O\oplus\O(2),D^V+[q],\nabla)\to (\O\oplus\O(2),\Omega_{A_\infty})$, that is a Higgs bundle, as desired.
\end{proof}
Let us now study the foliation in the coordinate $P^1$ in a neighbourhood of $A_\infty\cap B^0_\infty$. We remark that, since $\O_{\widetilde \M}(D)_{|\mathcal A_\infty}$ is trivial, the coordinate $P^1$ is globally defined.
\begin{oss}
	Thanks to the symmetry that exchanges $\mathcal{Q}^0$ and $\mathcal Q^\infty$, the same behaviour will occur near $A_\infty\cap B^\infty_\infty$.
\end{oss}
Let us consider the coordinate change $(t,q,p)\mapsto(r, q, P^1)$. We notice that $\mathcal A_\infty$ is defined by setting $r=0$. We get then the following vector field in restriction
\begin{align*}
	X^V_{\mathcal A_\infty}(q, P^1)&=\lim_{r\to0}(g_{\mathcal{B}_\infty^0*}X^V)(r, q,P^1)=\\&=-q(q-1)\left(2P^1+1\right)\partial_q-(q+1)P^1(P^1+1)\partial_{P^1}.
\end{align*}
that is vertical, as expected, for $q=0$. 
\begin{prop}
	A first integral for the isomonodromic foliation is given by 
	\[\frac{qP^1(P^1+1)}{(q-1)^2},\]
	that is the coefficient appearing in Proposition \ref{higgs}.
\end{prop}
\begin{proof}
		By the definition of first integral, we conclude via a direct computation.
\end{proof}
\begin{prop}
	The isomonodromic leaves in $\mathcal A_\infty$ are contained in the level sets of the first integrals and they form a pencil of elliptic curves.
\end{prop}
\begin{proof}
	The first part of the theorem is trivial. The Riemann-Hurwitz formula gives us the second part (as in Proposition \ref{prop:Hur}). It is indeed easy to see that each elliptic curve is tangent to the points $(T,q,P^1)=(0,1,0)$ and $(T,q,P^1)=(0,1, -1)$, that coincide with the intersection points between the two special sections in $\mathcal Q^1$ and the hypersurface $\mathcal A_\infty$. Their tangent slopes are 0 in both cases, meaning that the tangent lines are parallel to $P^1=0$, as shown in Figure \ref{fig:foliation}
\end{proof}

\begin{oss}
    The fibration given by the level sets of the first integral is not an Hitchin fibration. Indeed all the Higgs bundle on the same level set are the same.
\end{oss}

\medskip
\begin{center}
    \includegraphics[width=8cm]{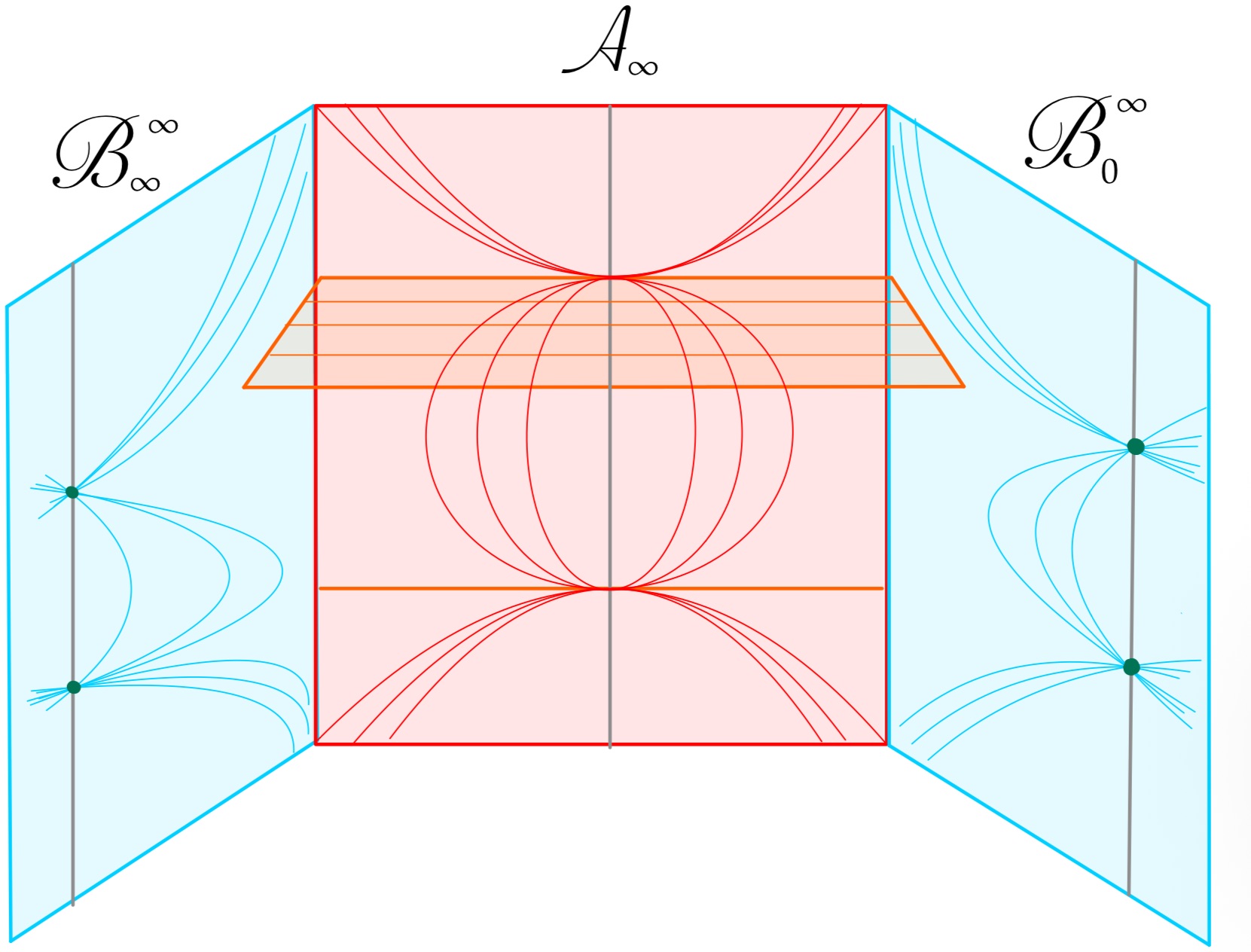}
    \captionof{figure}{}\label{fig:foliation}
\end{center}


\newpage
\bibliographystyle{plain}
\bibliography{bibliography}

\end{document}